\documentclass{article}
\usepackage[utf8]{inputenc}

\usepackage{mathtools}
\usepackage{amsmath}

\usepackage{amsthm}
\usepackage{dsfont}
\usepackage{xcolor}
\usepackage{amsfonts}
\newtheorem{theorem}{Theorem}[section]
\newtheorem{conjecture}[theorem]{Conjecture}

\newtheorem{remark}[theorem]{Remark}
\newtheorem{lemma}[theorem]{Lemma}
\newtheorem{proposition}[theorem]{Proposition}
\newtheorem{corollary}[theorem]{Corollary}
\usepackage{hyperref}

\newcommand{\nB}[2]{\mathcal{N}^{#1}(#2)}
\newcommand{\nOB}[2]{\overline{\mathcal{N}}^{#1}(#2)}
\newcommand{\NB}[2]{N^{#1}(#2)}
\newcommand{\NOB}[2]{\overline{N}^{#1}(#2)}
\newcommand{\E}[0]{\mathbb{E}}
\newcommand{\Prob}{\mathbb{P}}

\newcommand{\bbone}{\text{\usefont{U}{bbold}{m}{n}1}}
\MakeRobust{\bbone}

\DeclarePairedDelimiterX{\expectarg}[1]{[}{]}{%
  \ifnum\currentgrouptype=16 \else\begingroup\fi
  \activatebar#1
  \ifnum\currentgrouptype=16 \else\endgroup\fi
}

 \title{The number of particles at sublinear distances from the tip in branching Brownian motion}
\author{Gabriel Flath\thanks{Department of Statistics, University of Oxford, gabriel.flath@stats.ox.ac.uk}}

\begin{document}
\maketitle

\begin{abstract}
Consider a branching Brownian motion (BBM). It is well known \cite{Bramson1983ConvergenceOS, Lalley1987ACL} that the rightmost particle is located near \( m_t = \sqrt{2} t - \frac{3}{2\sqrt{2}} \log t \). Let $\mathcal{N}(t,x)$ be the set of particles within distance $x$ from $m_t$, where $x = o(t)$ grows with $t$. We prove that \(\#\mathcal{N}(t,x)/\pi^{-1/2}xe^{xm_t/t} e^{-x^2/(2t)} \) converges in probability to $Z_\infty$, the limit of the so-called derivative martingale, and that, for \( x = O( t^{1/3}) \), the convergence cannot be strengthened to an almost sure result. Moreover, we prove that the asymptotic overlap distribution of two particles sampled uniformly from $\mathcal{N}(t,x)$ converges to that of the critical derivative martingale measure. This establishes a universal genealogical picture of the BBM front at sublinear distances from the tip.
\end{abstract}

\section{Introduction and results}
\subsection{Introduction}

The (standard dyadic) Branching Brownian motion (BBM) is a classical stochastic process that describes the evolution of a population of particles undergoing both independent diffusion and branching. It is a staple of probability theory and statistical physics, with applications in e.g. population dynamics, reaction-diffusion equations, and random spin-glass models. Formally, BBM starts with a single particle at the origin, which diffuses as a standard Brownian motion. This particle lives for an exponentially distributed lifetime with mean one, and upon its death this particle splits into two offspring at the same position. Each daughter particle then starts an independent copy of the same process.
\\

The behaviour of the extremal particles of the BBM is by now a classical and well developed question which has seen some striking recent progress. Because the BBM is a toy model for log-correlated random fields, those developments have been instrumental in the study of the extremes of several models such as the Gaussian free field and the Random Energy Model universality class.\\

Let $\mathcal{N}_t$ be the collection of particles at time $t$. For $v \in \mathcal{N}_t$ and $s\in \left[0,t\right]$, we say that $u\in \mathcal{N}_s$ is its ancestor at time $s$ if $v$ is a descendant of $u$ and denote it $u\preceq v$. For $u\in \mathcal{N}_t$ and $s\in \left[0,t\right]$, we write $X_u(s) \in \mathbb{R}$  for the position of particle $u$ (or its ancestor) at time $s$. We define the branching Brownian motion process to be $X(t)=(X_u(t), u \in \mathcal{N}_t)$ and its natural filtration $\mathcal{F}_t=\sigma\{X(s), s\leq t \}$.\\

In a series of seminal works,  Bramson \cite{Bramson1983ConvergenceOS}, showed that the position of the rightmost particle in a BBM is centred around $m_t \coloneqq \sqrt 2 t -\frac3{2\sqrt 2 }\log t$, and more precisely that 
\[
\max_{u\in \mathcal{N}_t}X_u(t) - m_t \xrightarrow[t\rightarrow \infty]{d} W,
\]
where $W$ is a random variable whose distribution function is given by the so-called critical traveling wave solution to the Fisher-KPP equation. Lalley and Sellke \cite{Lalley1987ACL} then proved that  there exists a constant $C>0$ such that
\begin{equation}
     \Prob( \max_{u\in \mathcal{N}_t}X_u(t) - m_t\le x) \to \E\left[ e^{-C Z_\infty e^{-\sqrt 2 x}} \right],  
\end{equation} 
where  $Z_\infty>0$ is the almost sure limit of the derivative martingale:
\begin{equation*}
    Z_t=\sum_{u \in \mathcal{N}_t}(\sqrt{2}t - X_u(t))e^{\sqrt{2}(X_u(t)-\sqrt{2}t)}.
\end{equation*}
Observe that this suggests that
\begin{equation}
     \max_{u\in \mathcal{N}_t}X_u(t) - m_t -\log(CZ_\infty )/\sqrt{2}\xrightarrow[t\rightarrow \infty]{d} G,
\end{equation} 
where $G$ follows a Gumbel distribution, as later confirmed; see, e.g., \cite{adékon2012branching}.

To study the genealogy of the process, we will frequently appeal to the branching property. For any time $r \ge 0$ and any particle $u \in \mathcal{N}_r$, the descendants of $u$ form a branching Brownian motion shifted in time and space. We denote by $Z_\infty^{(u)}$ the almost sure limit of the derivative martingale associated with this shifted subtree. Conditional on $\mathcal{F}_r$, the random variables $(Z_\infty^{(u)})_{u \in \mathcal{N}_r}$ are independent and have the same distribution as the global limit $Z_\infty$.

The aim of our work is to extend the understanding of the front up to sublinear distance of $m_t$. In particular, we investigate the number of particles within a sublinear distance of $m_t$. Understanding this quantity provides deeper insights into the fine structure of the front and other existing results.\\

\subsection{Main results}\label{sub:mainresult}
For $t\geq 0$, define the set of particles to the right of $m_t-x$,
\begin{equation}
    \mathcal{N}(t,x) \coloneqq  \{u\in \mathcal{N}_t : X_u(t) \geq m_t -x\},
\end{equation}
and $N(t,x)\coloneqq \#\mathcal{N}(t,x)$ its cardinality.

\begin{theorem}\label{th:1}
    Let $x_t$ be such that, as $t\to \infty$, $x_t=o(t)$ and $x_t\rightarrow \infty$, then,
\begin{equation}
  \frac{ N(t,x_t)}{\pi^{-\frac{1}{2}}  x_t e^{\frac{m_t}{t}x_t - \frac{x^2}{2t}} } \xrightarrow[t\rightarrow \infty]{\mathbb{P}} Z_\infty.
\end{equation}

\end{theorem}
For convenience, we define  
\[
    f(t, x) \coloneqq \pi^{-1/2} x e^{ \frac{m_t}{t}x - \frac{x^2}{2t} }.
\]

\begin{remark}\label{re:rkconj}
The statement of Theorem \ref{th:1} is equivalent to, for any $\delta>0$,
    \begin{equation}
        \lim_{x\to \infty}\lim_{c\to 0}\sup_{t : x\leq ct}\Prob\left(\left|\frac{N(t,x)}{f(t,x)} -Z_\infty \right|>\delta \right) =0.
    \end{equation}    
\end{remark}

The following theorem establishes that almost sure convergence fails when \( x_t \) grows slowly with \( t \).
\begin{theorem}\label{thm1:2}
    Let \( x_t \) be such that \( 0\leq x_t \leq t^{1/3} \). Then,
    \begin{equation}
        \limsup_{t \to \infty} \frac{N(t, x_t)}{x_t e^{\sqrt{2}x_t}} = \infty \qquad \text{a.s.}
    \end{equation}    
\end{theorem}

The proof of Theorem~\ref{th:1} relies on bounding the number of pairs of particles that do not split early. Together with Theorem~\ref{th:1}, this bound yields an explicit limit overlap distribution for two particles sampled uniformly from \( \mathcal{N}(t, x) \).

Define the splitting time between two particles \( (u,v)\in \mathcal{N}_t^2 \) as
\[
    Q_t(u, v) \coloneqq \sup \left\{ s \leq t : \forall \gamma \leq s,\ X_u(\gamma) = X_v(\gamma) \right\}.
\]

\begin{theorem}\label{thm:ut}
     Let \( x_t \) be such that \( x_t \to \infty \), and \( x_t = o(t) \) as \( t \to \infty \). For $t\geq 0$, let
     \begin{equation}
         \mu_t \coloneqq \frac{1}{N(t,x_t)^2}\sum_{(u,v)\in \mathcal{N}(t,x_t)^2}\delta_{Q_t(u,v)} \quad (\text{setting } \mu_t \coloneqq \delta_0 \text{ if } N(t,x_t) = 0)
     \end{equation}
     the empirical overlap distribution for two particles sampled uniformly from $\mathcal{N}(t,x_t)$. Then $(\mu_t)_{t\ge0}$ converges weakly in probability to a limit $\mu_\infty$, that is,
    \begin{equation}
    \mu_t \xrightarrow[t\to\infty]{\mathbb{P}} \mu_\infty
    \quad \text{in the weak topology}.
    \end{equation}
    Where $\mu_\infty$ is characterized by 
    \begin{equation}
    \mu_\infty([r,\infty))
    =\frac{1}{Z_\infty^2}
    \sum_{u\in\mathcal{N}_r}
    \left(e^{\sqrt{2}X_u(r)-2r}\,Z_\infty^{(u)}\right)^2,
    \qquad \forall r\ge0,
    \quad \text{almost surely.}
    \end{equation}
\end{theorem}
Theorem \ref{thm:ut} directly implies the following.
\begin{corollary}
Let \( U \) and \( V \) be chosen uniformly at random from \( \mathcal{N}(t, x_t) \), and define \( Q_t = Q_t(U, V) \) if \( \mathcal{N}(t, x_t) \neq \emptyset \), and \( Q_t = 0 \) otherwise. Then \( (Q_t,Z_t)_{t \ge 0} \) converges in distribution to \( (Q_\infty,Z_\infty)\), where, conditional on $\mathcal{F}_\infty = \sigma\left(\cup_{t \ge 0} \mathcal{F}_t\right)$, $Q_\infty$ has distribution $\mu_\infty$.
\end{corollary}

\subsection{Discussion and related work}

First, observe that a first-moment calculation for $N(t,x)$ does not yield the correct order of magnitude. Indeed, a straightforward application of the many-to-One Lemma gives that
\begin{equation}
    \E[N(t,x)] \asymp t e^{\frac{m_t}{t}x - \frac{x^2}{2t}},
\end{equation}
whereas the convergence in Theorem \ref{th:1} suggests an order that is lower by a factor of $x/t$. This discrepancy arises from the correlation structure of branching Brownian motion. Specifically,  Arguin, Bovier, and Kistler \cite{Arguin2016} proved that extremal particles are unlikely to cross a certain barrier to the right. It is natural to expect that likewise, the particles contributing to $N(t,x)$ are similarly localised. The inflation of $\E[N(t,x)]$ can thus be attributed to the contribution of unlikely events where particles do manage to cross this barrier.\\

Theorem \ref{thm:ut}, Proposition \ref{pathlocup_proba} and Proposition \ref{prop:Kdist} provide the following description of the front of the branching Brownian motion. With high probability, for $x = o(t)$, most of the particles to the right of $m_t - x$ are in fact located within $m_t - x + O(1)$ at time $t$ and remain below the curve  $s\frac{m_t}{t}- (s\wedge(t-s))^\alpha$ for any $\alpha < 1/2$, throughout the time interval $[O(1), t - O(1)]$.  

More precisely, the trajectory of a particle in $\mathcal{N}(t,x)$, shifted by $m_ts/t$, resembles a Brownian bridge $B_s$ from distance $O(1)$ at time $r=O(1)$ and reaching $x$ at time $t$ conditioned to stay positive over the interval $[r, t]$.
\begin{itemize}
    \item For $x\ll \sqrt{t}$, $B_s$ behaves like a Brownian excursion (of length $t$) up to time $\gamma_\star =t-x^2$ and hence $B_{t-x^2}/x=O(1)$. After $\gamma_\star,$ the constraint to remain positive only contributes a constant factor, as this is a non-zero probability event. Furthermore, by Lemma \ref{bmbar}, the probability for a Brownian bridge starting at $O(1)$ at time $r$ and being at position $x$ at time $t-x^2$ without crossing 0 is of order $x/t$. 
    \item For $x\gg \sqrt{t}$, $B_s$ behaves like a Brownian excursion (of length $t$) up to time $\gamma_\star=(t/x)^2$ and hence $B_{(t/x)^2}x/t=O(1)$. After $\gamma_\star,$ the constraint to remain positive only contributes a constant factor, as this is a non-zero probability event. Furthermore, by Lemma \ref{bmbar}, the probability for a Brownian bridge starting at $O(1)$ at time $r$ and being at position $t/x$ at time $(t/x)^2$ without crossing 0 is of order $x/t$.
\end{itemize}

Thus, in both regimes, the probability of not hitting the barrier provides the correct multiplicative factor \( x/t \), as suggested at the beginning of this discussion.\\


Theorem \ref{thm:ut} implies that a pair of particles in the set $\mathcal{N}(t,x_t)$ typically splits early in the life of the process. It is interesting to compare this with the genealogy of the extremal particles, studied by Arguin, Bovier, and Kistler. In \cite[Theorem 2.1]{Arguin2016}, it is established that the branching time of a pair of extremal particles at time $t$ (i.e., within a compact distance of $m_t$) is typically either close to zero or close to $t$. When $x$ is fixed of order 1, there are only finitely many particles in $\mathcal{N}(t,x)$, and as they branch at a constant rate, there is a non-zero probability that two randomly sampled particles share a very recent common ancestor. In the language of the extremal point process, such closely related particles form a ``decoration''. If we condition two such extremal particles on having a small branching time (e.g., $\leq t/2$), they necessarily lie in different decorations. We expect that under this conditioning, their branching time also converges to our limit $Q_\infty$.

The key difference in our setting is as $x_t \to \infty$, the total number of decorations in $\mathcal{N}(t,x_t)$ tends to infinity, and no single recent branching event produces a positive proportion of the total population. Consequently, the probability of choosing two particles from the same decoration vanishes. Two randomly chosen particles will come from different decorations—and thus branched very far in the past— with probability tending to one, corresponding to the early-branching limit $Q_\infty$.\\

The limit measure $\mu_\infty$ in Theorem \ref{thm:ut} can be further characterized through its relationship with the \emph{critical measure} $\nu_\infty$ of the BBM. As established by Mallein \cite{Mallein2018}, the derivative martingale naturally induces a sequence of random probability measures $\nu_t$ on the particles at time $t$:
\begin{equation}
    \nu_t \coloneqq \frac{1}{Z_t} \sum_{u \in \mathcal{N}_t} (\sqrt{2}t - X_u(t)) e^{\sqrt{2}X_u(t) - 2t} \delta_u,
\end{equation}
which converges weakly almost surely to the critical measure $\nu_\infty$ on the boundary of the genealogical tree $\partial \mathbb{T}$. In particular if $V \sim \nu_\infty$, then the probability that $V$ is a descendant of some particle $u \in \mathcal{N}_r$ is given by the relative mass of the sub-tree rooted at $u$:
\begin{equation}
    \nu_\infty(u \preceq V) = \frac{e^{\sqrt{2}X_u(r) - 2r} Z_\infty^{(u)}}{Z_\infty}.
\end{equation}

Therefore, sampling $U,V$ independently from $\nu_\infty$,
the probability that their most recent common ancestor $Q(U, V)$ lived after time $r$ is given by the sum of the squared masses of the sub-trees at time $r$:
\begin{equation}
    \nu_\infty^{\otimes 2}(Q(U,V) \geq r) = \frac{1}{Z_\infty^2} \sum_{u \in \mathcal{N}_r} \left( e^{\sqrt{2}X_u(r) - 2r} Z_\infty^{(u)} \right)^2.
\end{equation}
This is exactly the formula that characterizes $\mu_\infty([r,+\infty))$ in Theorem \ref{thm:ut}. This identity reveals that the genealogical structure of the front is invariant across the sublinear regime, and that the ancestral splitting pattern governed by the derivative martingale is a universal feature for all particles at sublinear distance from the tip.\\

Note that, due to the phenomenon highlighted in the proof of Theorem \ref{thm1:2}, it is unlikely that the convergence in probability obtained in Theorem \ref{thm:ut} could be reinforced to an almost sure convergence for the regime $x_t \leq t^{1/3}$.\\
 
The extremal point process of the BBM has been studied extensively. In particular, Aïdékon, Berestycki, Brunet, and Shi~\cite{adékon2012branching} on the one hand, and Arguin, Bovier, and Kistler~\cite{arguin2011extremal} on the other, independently established that, viewed as a point process and shifted by $m_t$, the BBM converges in distribution to a limit point process called the \emph{extremal point process}. More precisely, let
\begin{equation}
    \mathcal{E}_t = \sum_{u \in \mathcal{N}_t} \delta_{X_u(t) - m_t}.
\end{equation}
As $t \to \infty$, they showed that the family of point processes $\mathcal{E}_t$ converges in distribution to a limit, sometimes referred to as a randomly shifted decorated Poisson point process (SDPPP).

Using the convergence of $\mathcal{E}_t$, Cortines, Hartung, and Louidor study the asymptotic number of particles in the extremal point process. They state, \cite[Theorem 1.1]{cort}, that there exists $C_\star>0$ such that
\begin{equation}\label{eq:erdf}
     \frac{\mathcal{E}\left([-x,\infty)\right)}{C_\star Z_\infty xe^{\sqrt{2}x}} \xrightarrow[x\to \infty]{\mathbb{P}}  1.
\end{equation}
 
Since $\mathcal{E}$ is the limit of $\mathcal{E}_t$ and $\mathcal{E}_t\left([-x,\infty)\right)=N(t,x)$, the convergence in \eqref{eq:erdf} is a direct consequence of our results. Specifically, letting $t\to \infty$ first in the Remark \ref{re:rkconj}, we obtain that for any $\epsilon>0$,
\begin{equation}
  \lim_{x\to \infty}\lim_{t\to \infty}P\left(\left| \frac{\mathcal{E}_t\left([-x,\infty)\right)}{C_\star Z_\infty xe^{\sqrt{2}x}} -1 \right| > \epsilon \right) = 0.
\end{equation}
This recovers \eqref{eq:erdf},  yielding the explicit value of the constant $C_\star$. Interestingly, if $x_t=o(\sqrt{t})$, the denominator in Theorem \ref{th:1}  depends only on $x$ and satisfies $f(t,x) \sim \pi^{1/2} xe^{\sqrt 2 x}$, which matches the growth of the extremal point process.\\ 

In \cite{ma2024doublejumpmaximumtwotype}, Ren and Ma, leveraging previous methods developed in \cite{shi1} and \cite{Pain1}, proved a  law of large numbers for a particle chosen according to the Gibbs measure in \cite[Proposition 2.6]{ma2024doublejumpmaximumtwotype}. Plugging in the following value we can link their result to our Theorem \ref{th:1}. Let $K>0$ and $x_t$ be such that $\log(t)^3\ll x_t\leq \sqrt{t}$ and plugging in $G(z)=\bbone_{[-K,0]}(z)e^{\sqrt{2}z}$, $r_t=(x_t+\frac{3}{2\sqrt{2}}\log(t))t^{-1/2} $, $h_t=t^{-1/2}$, one obtain that
\begin{equation}\label{eq:renma}
  \frac{ N(t,x_t)-N(t,x_t-K)}{f(t,x_t) } \xrightarrow[t\rightarrow \infty]{\mathbb{P}} Z_\infty \int_0^K \sqrt{2}e^{-\sqrt{2}z}dz=\left(1-e^{-\sqrt{2}K}\right)Z_\infty.
\end{equation}

The convergence in \eqref{eq:renma} is also a direct consequence of Theorem \ref{th:1}. Thus the results are consistent for the regime $x_t \in [\log(t)^{3+\epsilon}, \sqrt{t}]$. Note that Theorem \ref{th:1} extends their result by considering all the particles to the right of $m_t-x$ (i.e. $K$ in \eqref{eq:renma} could depend on $t$) and extending the interval of $x_t$ both near the tip for which their proof fails, and above $\sqrt{t}$ for which the Gaussian term $e^{-x^2/t}$ in $f(t,x)$ is non-negligible. One can also obtain a functional result from our work since we prove uniform convergence of the density of particles. \\

Our study is also related to several other results in the literature. In \cite{Pain1}, Pain studies the convergence of the Gibbs measure. One could wonder whether the convergence in probability in \cite[Theorem 1.1]{Pain1} can be strengthened  to an almost sure one. Despite the specific weighting of particles in the Gibbs measure, Corollary 1.3 in \cite{Pain1} establishes that the particles supporting the Gibbs measure with parameter $1+x/t$ are located around $m_t-x$ .  Therefore, for the regime (i) in \cite[Theorem 1.1]{Pain1}, reinforcing the convergence is not possible due to the mechanism exhibited in Theorem \ref{thm1:2}. Moreover, \cite[Theorem 1.1, (iv)]{Pain1} could help relax the condition on $R$ in Proposition \ref{espcondr}, potentially paving the way to proving, for $x\gg\sqrt{t}$, the almost sure convergence of $N(t,x)$, and  establishing a connection between $N(t,x)$ and $Z_{(t/x)^2}$.\\


The regime $x_t \asymp t$ (the number of particles at a linear distance from the tip) has been studied by Biggins in \cite{Biggins} for the branching random walk and by Glenz, Kistler, and Schmidt in \cite{glenz2017mckeans} for the branching Brownian motion. It is worth noting that our exact normalization $f(t,x) \propto x e^{ \frac{m_t}{t}x - \frac{x^2}{2t}}$ is algebraically consistent with these results. Evaluating particles at a distance $ct$ from the leading order tip $\sqrt{2}t$ corresponds to a distance $x = ct - \frac{3}{2\sqrt{2}}\log t$ relative to $m_t$. Substituting this into $f(t,x)$, the cross-terms in $c \log t$ cancel, recovering the classical scaling factor $t^{-1/2} \exp\left( (\sqrt{2}c - c^2/2)t \right)$ of the linear regime. However, while this deterministic normalization extrapolates to the linear regime, the limiting random variable fundamentally differs: in the bulk, the limit is driven by the additive martingale, whereas for any sublinear distance $x = o(t)$, the fluctuations are governed by the derivative martingale $Z_\infty$.

Finally, we note that in the physics literature, the density of particles at a given distance of $m_t$ has been studied by Brunet and Derrida in the context of BBM with selection \cite{PhysRevE.56.2597}, and by Munier for standard BBM \cite{Le_2022}.\\

\subsection{A conjecture}

Theorem \ref{thm1:2} establishes that, for $x_t \leq t^{1/3}$, the renormalised number of particles $N(t,x_t)/f(t,x_t)$ does \emph{not} converge almost surely. This stems from potential fluctuations at the front around times \( t - O(1) x_t^2 \).

By leveraging the techniques developed in this paper and using the barrier $\beta(s) = \sqrt{2}s$ for $s \geq r$, one can strengthen the convergence in Proposition \ref{th:genealogy} to an almost sure convergence via a Borel–Cantelli argument. Specifically, in the regime \( x_t \gg \sqrt{t} \) and \( R_t \gg \log(t)(t/x)^2 \), we obtain:

\begin{equation}\label{eq:gihfdg}
   \frac{\#\left\{(u,v) \in \mathcal{N}(t,x_t)^2 : Q_t(u,v) \geq R_t \right\}}{f(t,x_t)^2} \xrightarrow[t\to \infty]{} 0 \quad \text{a.s.}
\end{equation}

Crucially, in Proposition \ref{espcondr}, we require \( R_t = o(\sqrt{t}) \) and \( R_t = o(t/x_t) \), while Proposition \ref{th:genealogy} only requires \( R_t \to \infty \), allowing these two results to be applied jointly. However, the almost sure convergence in \eqref{eq:gihfdg} is currently proved only for \( R_t \gg \log(t)(t/x_t)^2 \). Therefore, to extend Theorem \ref{th:1} to an almost sure statement in the regime \( x_t \gg \sqrt{t} \), one must either broaden the admissible range for \( R_t \) in \eqref{eq:gihfdg} or relax the conditions on \( R_t \)  in Proposition \ref{espcondr}. This could be achieved by showing that a stronger barrier only removes particles that contribute negligibly to $N(t,x_t)$ almost surely.\\

We conjecture that the convergence in Theorem \ref{th:1} indeed holds almost surely for $x_t \gg \sqrt{t}$.
\begin{conjecture}
Let $x_t$ be such that, as $t \to \infty$, $x_t \geq \sqrt{t} \log(t)^{1+\epsilon}$ and $x_t = o(t)$. Then,
\begin{equation}
  \frac{N(t,x_t)}{f(t,x_t)} \xrightarrow[t \to \infty]{} Z_\infty \quad \text{a.s.}
\end{equation}
\end{conjecture}

\subsection{Acknowledgments}
I am deeply grateful to my PhD advisor, Julien Berestycki, for his guidance, support, and the many insightful discussions throughout the development of this work, as well as for his helpful suggestions on the writing of the paper. I would particularly like to thank Louis Chataignier for his suggestions and for discussing several improvements, notably regarding the genealogy part of the paper. Finally, I thank the many colleagues at various conferences for the stimulating conversations that contributed to this work.

\section{Proof of the main results}\label{sec:proofff}
In this section, we prove the theorems stated in Subsection~\ref{sub:mainresult}. We begin by establishing Theorem~\ref{th:1}, proceeding via several intermediate results concerning the path localisation and genealogy of the particles. The proofs of these intermediate results are deferred to Section~\ref{Sec:proofprop}. We then turn to the proof of Theorem~\ref{thm:ut} and Theorem~\ref{thm1:2}.

To prove Theorem \ref{th:1}, we essentially follow the approach pioneered in \cite{arguin2012ergodic}: the key steps are first to show that most of the particles that contribute to $N(t,x)$ have a well-localised path. An extra technicality in this first step is to prove that, since $x\to \infty$, the particles contributing to $N(t,x)$ are actually near $m_t -x +O(1)$. We then use a concentration argument to show that the number of such particles is close to its conditional expectation given the initial behaviour of the branching Brownian motion. Note that this is precisely the strategy employed in \cite{glenz2017mckeans}. However, as usual, adapting it to the critical case involves some subtleties.

Let us first introduce some notation related to barrier events. For $a<b$, denote
\[
\mathcal{N}(t,[a,b]) = \left\{u \in \mathcal{N}_t: m_t - X_u(t) \in [a,b] \right\}.
\]
For an interval $I \subset [0,t]$ and a function $\beta \colon I \to \mathbb{R}$, denote by $\mathcal{H}^\beta(t,I)$ the set of particles which remain under the curve $s \mapsto \beta(s)$ during the time interval $I$,
\[
\mathcal{H}^\beta(t,I) \coloneqq \left\{u\in \mathcal{N}_t : X_u(s)\leq \beta(s) \; \forall s \in I \right\}.
\]
We further define
\[
\nB{\beta}{t, I,x} \coloneqq \mathcal{N}(t,x) \cap \mathcal{H}^\beta(t, I), \quad
\nB{\beta}{t, I,[a,b]} \coloneqq \mathcal{N}(t,[a,b]) \cap \mathcal{H}^\beta(t, I),
\]
the set of particles above $m_t - x$ (in $m_t - [a,b]$) at time $t$ which remain under $\beta$ during $I$. Similarly,
\[
\nOB{\beta}{t,I,x} \coloneqq \mathcal{N}(t,x) \setminus \mathcal{H}^\beta(t,I), \quad
\nOB{\beta}{t,I,[a,b]} \coloneqq \mathcal{N}(t,[a,b]) \setminus \mathcal{H}^\beta(t,I),
\]
denote the sets of particles above $m_t - x$ (in $m_t - [a,b]$) at time $t$ which cross $\beta$ during $I$. For each of the sets we have defined, we denote its cardinality by the corresponding capital letter; for example, $\NB{\beta}{t, I,x} \coloneqq \#\nB{\beta}{t, I,x}$. Finally, define the family of functions $\beta_t(s) \coloneqq s\frac{m_t}{t}$ and $\Gamma_{t,\alpha}(s) \coloneqq (s \wedge (t-s))^\alpha$ for $\alpha \geq 0$.\\

First, we state that for large $K$, the normalised quantity $N(t, x-K)/f(t,x)$ becomes negligible in the limit as $x \to \infty$.
\begin{proposition}\label{prop:Kdist}For any $\delta>0$,
\begin{equation}
    \lim_{K\to\infty} \lim_{x\to\infty}\sup_{t\geq x}\Prob\left(\frac{N(t,x-K)}{f(t,x)} > \delta\right) = 0.
\end{equation}

\end{proposition}

The following proposition allows us to only consider particles whose trajectories remain below a prescribed barrier curve.
\begin{proposition}\label{pathlocfin}
Let $\alpha\in [0,1/2)$, $\lambda>\frac{1}{2}$, $\rho_{t,r}(s)= \beta_t(s)-\bbone_{s\in [r,t-r]}\Gamma_{t,\alpha}(s)$. For any $\delta>0$,
\begin{equation}
    \lim_{r\to \infty}\sup_{t,x: t>3r,  r^{\lambda}\leq x\leq t}\Prob\left( \frac{\NOB{\rho_{t,r}}{t, [r,t],x} }{f(t,x)} >\delta\right) =0.
\end{equation}
    
\end{proposition}

Propositions \ref{prop:Kdist} and \ref{pathlocfin} are proven in Subsection \ref{subSec:proofprop1}\\

The following proposition establishes the almost sure convergence of the normalised conditional expectation for localised particles to \( Z_\infty \), under asymptotic conditions on suitably scaled parameters.
\begin{proposition}\label{espcondr}
Let $x_t, K, r, R$ be such that, as $t\to \infty$, \( x_t \to \infty \) with \( x_t = o(t)\), \( K \to \infty \), \( r \to \infty \), \( R \to \infty \) with \( R = o(\sqrt{t}) \), \( R = o(t/x_t) \), and \( R \geq r \). Then the following holds:
\begin{equation}\E\left[ \NB{\beta_t}{t, [r,t], [x_t-K,x_t]}\mid \mathcal{F}_{R} \right]=  f(t,x_t)Z_R(1+o_t(1)) \quad \text{a.s,} 
\end{equation}
where \( o_t(1) \to 0 \) almost surely as \( t \to \infty \).
\end{proposition}
The proof of Proposition~\ref{espcondr} is given in Subsection~\ref{subSec:proofprop2}.\\

We now state results to quantify the difference between the normalised conditional expectation and the actual number of localised particles. The following lemma bounds the mean squared difference between the number of localised particles and its conditional expectation at time $R$, in terms of the number of pairs of particles that branch after time $R$.
\begin{lemma}\label{lemmaL2bound}
For parameters \( t \geq 0 \), \( R \in [0,t] \), \( X \subset \mathbb{R} \), \( I \subset [0,t] \) and a function \( \beta \colon I \to \mathbb{R} \):
\begin{equation}
        \begin{split}
            \E&\left[ \left(N^{\beta}(t, I,X) - \E\left[N^{\beta}(t, I,X) \mid \mathcal{F}_R \right] \right)^2 \right] \\
            &\qquad \leq \E\left[\#\left\{(u,v) \in \mathcal{N}^\beta(t,I,X)^2 : Q_t(u,v) \geq R \right\}\right].
        \end{split}
\end{equation}
\end{lemma}

Finally, the expected number of $\beta_t$-localised particle pairs branching after time $R$ is asymptotically negligible relative to $f(t,x)^2$:
\begin{proposition}\label{L2boundlocancestor}
Let $x_t, K, r, R_t$ be such that, as $t\to \infty$, \( x_t \to \infty \) with \( x_t = o(t) \), \( R_t \to \infty \), $2r\leq R_t$, $3r \leq t$ and $K\in \mathbb{R}$. Then the following holds:

\begin{equation}
    \frac{\E\left[\#\left\{(u,v) \in \mathcal{N}^{\beta_t}(t,[r,t],[x_t-K,x_t])^2 : Q_t(u,v) \geq R_t \right\}\right]}{f(t,x_t)^2}\xrightarrow[t\rightarrow \infty]{}0.
\end{equation}
\end{proposition}
The proofs of Lemma~\ref{lemmaL2bound} and Proposition~\ref{L2boundlocancestor} are given in Subsection~\ref{subSec:proofprop3}.\\
We are now ready to prove Theorem \ref{th:1}.
\begin{proof}[Proof of Theorem \ref{th:1}]
Let \( I = [r, t] \) and \( \beta_t(s) = s\frac{m_t}{t} \). For any \( r, R, K > 0 \), we have the decomposition:
\begin{equation}\label{decompproba}
\begin{split}
    N(t,x) &= N(t,x-K) + \NOB{\beta_t}{t, I,[x-K,x]} \\
    &\quad + \E\left[N^{\beta_t}\left(t,I,[x-K,x]\right) \mid \mathcal{F}_{R} \right] \\
    &\quad + \left(N^{\beta_t}(t, I,[x-K,x]) - \E\left[N^{\beta_t}(t, I,[x-K,x]) \mid \mathcal{F}_{R} \right]\right).
\end{split}
\end{equation}
Therefore,
\begin{equation}\label{decompproba2}
\begin{split}
    \Prob&\left(\left|\frac{N(t,x)}{f(t,x)} -Z_\infty \right|>\delta\right) \leq \Prob\left(\left|\frac{N(t,x-K)}{f(t,x)}\right|>\frac{\delta}{4}\right) 
    \\& \qquad+\Prob\left(\left|\frac{\NOB{\beta_t}{t, I,[x-K,x]} }{f(t,x)}\right|>\frac{\delta}{4}\right) 
    \\&\qquad+\Prob\left(\left|\frac{\E\left[N^{\beta_t}(t,I,[x-K,x])\mid \mathcal{F}_{R} \right]}{f(t,x)} -Z_\infty \right|>\frac{\delta}{4}\right)
    \\&\qquad+\Prob\left(\left(\frac{N^{\beta_t}(t, I,[x-K,x]) - \E\left[N^{\beta_t}(t, I,[x-K,x])\mid \mathcal{F}_{R} \right]}{f(t,x)}\right)^2>\left(\frac{\delta}{4}\right)^2 \right).
\end{split}
\end{equation}

We aim to demonstrate that the LHS of \eqref{decompproba2} converges to zero as $t\to \infty$ for $x=x_t$ satisfying $x_t\rightarrow \infty$ and $x_t=o(t)$. Since the LHS is independent of $r,R,K$, we choose those quantities as functions of $t$ diverging to infinity with $t$ in the most convenient way:  
\begin{align*}
    R&=R_t  \text{ such that } R=o(\sqrt{t}) \text{ and } R=o(t/x_t), \\
    r&=r_t \text{ such that } 3r_t\leq t, \, r^{\lambda}\leq x_t \text{ for a } \lambda>\frac{1}{2}\text{ and } r_t=o(R_t),\\
    K&=K_t \text{ such that } K_t=o(r_t).
\end{align*}

We now prove that each term on the RHS of \eqref{decompproba2} converges to zero as $t\to \infty$.
The first term on the RHS converges to 0 by Proposition \ref{prop:Kdist}. Next, observe that 
\begin{equation}
    \Prob\left(\left|\frac{\NOB{\beta_t}{t, I,[x-K,x]}  }{f(t,x)}\right|>\frac{\delta}{4}\right)\leq \Prob\left(\left|\frac{\NOB{\beta_t}{t, I,x} }{f(t,x)}\right|>\frac{\delta}{4}\right),
\end{equation}
thus the second term on the RHS converges to 0 by Proposition \ref{pathlocfin}.
The third term on the RHS converges to 0 by Proposition \ref{espcondr}.
Finally, by Markov's inequality and Lemma \ref{lemmaL2bound},
\begin{equation}
\begin{split}
    \Prob&\left(\left(\frac{N^{\beta_t}(t, I,[x-K,x]) - \E\left[N^{\beta_t}(t, I,[x-K,x])\mid \mathcal{F}_{R} \right]}{f(t,x)}\right)^2>\left(\frac{\delta}{4}\right)^2 \right)
    \\& \leq \frac{\mathbb{E}\left[\left(N^{\beta_t}(t, I,[x-K,x]) - \E\left[N^{\beta_t}(t, I,[x-K,x])\mid \mathcal{F}_{R} \right]\right)^2\right]}{(4^{-1}\delta f(t,x))^2}
    \\& \leq \frac{\E\left[\#\left\{(u,v) \in \mathcal{N}^{\beta_t}(t,I,[x-K,x])^2 : Q_t(u,v) >R \right\}\right]}{(4^{-1}\delta f(t,x))^2}.
\end{split}
\end{equation}
Thus, by Proposition \ref{L2boundlocancestor}, the last term on the RHS converges to 0.

\end{proof}

Having established a bound on the number of pairs of localized particles that do not split early, we derive the following intermediate result: two particles chosen uniformly at random from $\mathcal{N}(t,x_t)$ must have branched early. This translates to the tightness in probability of the random measures $\mu_t$, which will serve as a key ingredient for Theorem~\ref{thm:ut}.

\begin{proposition}\label{th:genealogy}
    Let \( x_t \) and \( R_t \) be such that \( x_t \to \infty \), \( x_t = o(t) \), and \( R_t \to \infty \) as \( t \to \infty \). Then for any \( \delta > 0 \),
    \begin{equation}
        \lim_{t \to \infty} \mathbb{P} \left( \mu_t([R_t, \infty)) \left( \frac{N(t, x_t)}{f(t,x_t)} \right)^2 > \delta \right) = 0.
    \end{equation}

    As a consequence of Theorem~\ref{th:1}, this implies that the family of random measures $(\mu_t)_{t \ge 0}$ is tight in probability; that is, for any $\epsilon > 0$,
    \begin{equation}
        \lim_{R \to \infty} \sup_{t\geq 0} \mathbb{P} \left( \mu_t([R, \infty)) > \epsilon \right) = 0.
    \end{equation}
\end{proposition}

\begin{proof}[Proof of Proposition \ref{th:genealogy}]

In this proof, we say that a particle $u\in \mathcal{N}_t$ is \emph{localised} (and write $u \text{ loc}$) during $[r, t ]$ if and only if
\begin{equation}
    X_u(s) \leq s\frac{m_t}{t} \quad \forall s \in [r,t].
\end{equation}
Otherwise, we say that it is \emph{not localised}.

As in the statement of the theorem, let $x=x_t$ and $R_t$  be such that, as $t\to \infty$, $x_t=o(t)$, $x_t\rightarrow \infty$ and $R_t\rightarrow \infty$.
Here, we choose $r=r_t$ and $K=K_t$, diverging to $\infty$ with $t$, satisfying $3r_t\leq t, \, r^{\lambda}\leq x_t \text{ for a } \lambda>\frac{1}{2}, r_t=o(R_t)$, and $K_t=o(r_t)$.  

We aim to control the quantity
\begin{equation}
     \mu_t([R_t, \infty)) \left( \frac{N(t, x_t)}{f(t,x_t)} \right)^2 = \frac{\#\big\{(u,v) \in \mathcal{N}(t,x)^2 : Q_t(u,v) \ge R_t \big\}}{f(t,x)^2}.
\end{equation}

We begin by localising particles after time \( r \). Observe that:
\begin{equation}
\begin{split}
    \#\big\{(u,v) \in \mathcal{N}(t,x)^2 :& \ u \text{ or } v \text{ not loc on } [r,t] \big\}\\& \leq 2\#\left\{(u,v) \in \mathcal{N}(t,x)^2 : u \text{ not loc on } [r,t] \right\}
    \\& =2N(t,x)\#\left\{u \in \mathcal{N}(t,x) : u \text{ not loc on } [r,t]\right\}.
\end{split}
\end{equation}
Therefore by Theorem \ref{th:1} and Proposition \ref{pathlocfin} (with $\alpha=0$), for any $\delta>0$,
\begin{equation}
    \Prob\left(\frac{\#\left\{(u,v) \in \mathcal{N}(t,x)^2 : u \text{ or } v \text{ not loc on } [r,t] \right\}}{f(t,x)^2}>\delta \right)\xrightarrow[t\rightarrow \infty]{} 0.
\end{equation}
Similarly, by Proposition~\ref{prop:Kdist},
\begin{equation}
\Prob\left(
\frac{
\#\Big(
\mathcal{N}(t,x-K)\times \mathcal{N}(t,x)
\;\cup\;
\mathcal{N}(t,x)\times \mathcal{N}(t,x-K)
\Big)
}{f(t,x)^2}
> \delta
\right)
\xrightarrow[t\to\infty]{} 0 .
\end{equation}
Moreover, by Markov's inequality and Proposition \ref{L2boundlocancestor},
\begin{equation}
\begin{split}
    \Prob&\left(\frac{\#\left\{(u,v) \in \mathcal{N}(t,[x-K,x])^2 : u \text{ and } v \text{ loc on } [r,t] \text{ and } Q_t(u,v) \geq R \right\}}{f(t,x)^2}>\delta \right)
    \\&\leq \frac{\E\left[\#\left\{(u,v) \in \mathcal{N}^{\beta_t}(t,I,[x-K,x])^2 : Q_t(u,v) \geq R \right\}\right]}{\delta f(t,x)^2}\xrightarrow[t\rightarrow \infty]{}0.
\end{split}
\end{equation}
Thus proving Proposition \ref{th:genealogy}.
\end{proof}

To convert this tightness into the weak convergence stated in Theorem~\ref{thm:ut}, we rely on the following general criterion for random probability measures.

\begin{lemma}[Characterization of Weak Convergence]\label{lem:characterization}
    Let $(\nu_t)_{t \geq 0}$ be a family of random probability measures on $\mathbb{R}_+$. Suppose the following two conditions are satisfied:
    \begin{enumerate}
        \item[(i)] For every $r \geq 0$, there exists a random variable $F_r$ such that
        \begin{equation}
            \nu_t([r, \infty)) \xrightarrow[t \to \infty]{\mathbb{P}} F_r.
        \end{equation}
        \item[(ii)] The family is tight in probability; that is, for every $\epsilon > 0$,
        \begin{equation}
            \lim_{M \to \infty} \limsup_{t \to \infty} \mathbb{P}\left( \nu_t([M, \infty)) > \epsilon \right) = 0.
        \end{equation}
    \end{enumerate}
    Then, there exists a random probability measure $\nu$ such that $\nu_t$ converges weakly in probability to $\nu$ as $t \to \infty$. The limit $\nu$ is uniquely characterized by the property that, almost surely, $\nu([r, \infty)) = F_r$ for every continuity point $r$ of the mapping $x \mapsto \nu([x, \infty))$.
\end{lemma}

The proof of Lemma \ref{lem:characterization} relies on a standard diagonal extraction argument adapted from the proof of Helly's Selection Theorem (see, e.g., \cite[Theorem 25.9]{Billingsley}).

With this lemma and Proposition \ref{th:genealogy} in hand, the proof of Theorem \ref{thm:ut} follows naturally.

\begin{proof}[Proof of Theorem \ref{thm:ut}]
    By Proposition \ref{th:genealogy}, the family of random measures $(\mu_t)_{t \geq 0}$ is tight in probability, which immediately satisfies condition (ii) of Lemma \ref{lem:characterization}. To establish weak convergence, it remains to verify condition (i).

    Fix $r \geq 0$. The event $\{Q_t(u,v) \geq r\}$ implies that $u$ and $v$ descend from a common ancestor $w \in \mathcal{N}_r$ alive at time $r$. Partitioning the pairs in $\mathcal{N}(t,x_t)$ by their ancestor at time $r$ yields
    \begin{equation}
        \mu_t([r, \infty)) = \sum_{w \in \mathcal{N}_r} \left( \frac{N_w(t, x_t)}{N(t, x_t)} \right)^2,
    \end{equation}
    where $N_w(t, x_t) \coloneqq \# \{ v \in \mathcal{N}(t, x_t) : w \preceq v \}$. 

    By Theorem \ref{th:1},
    \begin{equation}
        \frac{N(t, x_t)}{f(t, x_t)} \xrightarrow[t \to \infty]{\mathbb{P}} Z_\infty.
    \end{equation}
    To evaluate $N_w(t, x_t)$, we apply Theorem \ref{th:1} conditionally on $\mathcal{F}_r$ to the shifted subtree rooted at $w$. The target spatial boundary $m_t - x_t$ for these descendants, when viewed from the particle's position $X_w(r)$ and over the remaining time $t-r$, corresponds to a relative target of $m_{t-r} - (x_t + X_w(r) - \sqrt{2}r) + o(1)$. This yields 
    \begin{equation}
        \frac{N_w(t, x_t)}{f\big(t-r, x_t + X_w(r) - \sqrt{2}r\big)} \xrightarrow[t \to \infty]{\mathbb{P}} Z_\infty^{(w)}.
    \end{equation}

    Taking the ratio, the polynomial and Gaussian terms in $f$ cancel asymptotically as $t \to \infty$, leading to
    \begin{equation}
        \frac{N_w(t, x_t)}{N(t, x_t)} \xrightarrow[t \to \infty]{\mathbb{P}} \frac{Z_\infty^{(w)}}{Z_\infty} e^{\sqrt{2}X_w(r)-2r}.
    \end{equation}
    Since the population $\mathcal{N}_r$ is almost surely finite, the finite sum converges in probability:
    \begin{equation}
        \mu_t([r, \infty)) \xrightarrow[t \to \infty]{\mathbb{P}} F_r \coloneqq \frac{1}{Z_\infty^2} \sum_{w \in \mathcal{N}_r} \left( e^{\sqrt{2}X_w(r)-2r} Z_\infty^{(w)} \right)^2.
    \end{equation}
    This establishes condition (i) of Lemma \ref{lem:characterization}. 
    
    To conclude the proof, observe that the process $r \mapsto F_r$ is piecewise constant, with jumps occurring exclusively at the branching times of the BBM. Since the branching times are continuous random variables, the probability of a jump occurring exactly at a fixed deterministic $r$ is zero. Therefore, $\mathbb{P}(\mu_\infty(\{r\}) > 0) = 0$ for all $r \ge 0$, implying that every $r$ is almost surely a continuity point of the limit measure $\mu_\infty$. By Lemma \ref{lem:characterization}, $(\mu_t)_{t \ge 0}$ converges weakly in probability to $\mu_\infty$, uniquely characterized by $\mu_\infty([r, \infty)) = F_r$ almost surely.
\end{proof}

We now prove Theorem \ref{thm1:2}.
\begin{proof}[Proof of Theorem \ref{thm1:2}]
Roberts \cite{Roberts_2013} established the following result for the position $M_t$ of the rightmost particle of the BBM 
\begin{equation}
    \limsup_{t\to \infty}\frac{M_t - \sqrt{2}t}{\log(t)}=\frac{1}{2\sqrt{2}}.
\end{equation}
The analogous result for the branching random walk was previously established by Hu and Shi \cite{hushi}. Moreover, in \cite{Hu2012TheAS}, Hu derived integral tests describing the almost sure lower limits for the extremal particle position in the branching random walk; the same arguments apply to BBM. In particular, he established that 

\begin{equation}\label{eq:Husup}
    \limsup_{t\to \infty} M_t -\sqrt{2}t -\tfrac{\log(t)}{2\sqrt{2}}= +\infty.
\end{equation}

Let $x_t$ be such that $x_t\leq t^{1/3}$. We will show that, for any $K$, there exists $\epsilon>0$, a sequence of stopping times $(t_n)_{n\in \mathbb{N}}$ and $(T_n)_{n\in \mathbb{N}}$ satisfying $t_n\leq T_n\leq t_{n+1}$ for all $n\in \mathbb{N}$,  such that for all sufficiently large $n$,
\begin{equation}
    P\left(N(T_n,x_{T_n})\geq Kf(T_n,x_{T_n}) \mid \mathcal{F}_{t_n}\right) >\epsilon.
\end{equation}
It follows from the conditional version of the Borel–Cantelli lemma that
\begin{equation}
    \limsup_{n\to \infty}\frac{N(T_n,x_{T_n})}{f(T_n,x_{T_n})} = \infty,
\end{equation}
thereby establishing the desired result.\\

Recall that \( Z_\infty > 0 \) almost surely. Therefore, by Theorem \ref{th:1}, There exist $X>0$ and $\delta>0$ such that for all $x\geq X$ and $t\geq x^2/4$,
\begin{equation}\label{jio4}
    P\left(N(t,x)\geq \delta f(t,x) \right) >\epsilon.
\end{equation}

Let $V=\max\left((\log(e^{5/2}K)-\log(\delta))/{\sqrt{2}},X\right)$. By \eqref{eq:Husup}, there exists a sequence $t_n \to \infty$ of stopping times such that, for all $n \in \mathbb{N}$, $t_{n+1}\geq 2t_n$ and
\begin{equation}
    M_{t_n}\geq \sqrt{2}t_n-\frac{\log(t_n)}{2\sqrt{2}} + V=m_{t_n}+\frac{\log(t_n)}{\sqrt{2}} + V.
\end{equation} 
Moreover define $T_n\coloneqq t_n+t^{\frac{2}{3}}_n$, for all $n\in \mathbb{N}$.

Let \( u_t^\star \) denote the rightmost particle at time \( t \), and for \( s \in [0, t] \), define \( N_s(t, x) \) as the number of particles at time \( t \) to the right of \( m_t - x \) that are descendants of \( u_s^\star \):
\begin{equation}
    N_s(t,x) =\# \{u\in \mathcal{N}(t,x) : u^\star_s\leq u \}.
\end{equation}
Observe that $N_s(t,x)\leq N(t,x)$ and that:
\begin{equation}\label{gi49}
    m_{t^{2/3}_n} -V \geq m_{T_n} - m_{t_n} - \frac{\log(t_n)}{\sqrt{2}} -V.
\end{equation}

Therefore, for all $n$ such that $t_n^{1/3}\geq V $,
\begin{align}
    &P\left(N(T_n,x_{T_n})\geq Kf(T_n,x_{T_n}) \mid \mathcal{F}_{t_n}\right) 
    &&  \notag \\
    &\quad \geq P\left(N_{t_n}(T_n,x_{T_n})\geq Kf(T_n,x_{T_n}) \mid \mathcal{F}_{t_n} \right)
    &&  \notag \\
    &\quad \geq P\left(N(t^{\frac{2}{3}}_n,x_{T_n}+V)\geq Kf(T_n,x_{T_n}) \right)
    && \qquad \text{by \footnotesize \eqref{gi49}} \notag \\
    &\quad \geq P\left(N(t^{\frac{2}{3}}_n,x_{T_n}+V)\geq e Kf( t^{\frac{2}{3}}_n ,x_{T_n}) \right)
    && \qquad \text{\footnotesize (\( e^{x_{T_n}^2/(2t_n^{2/3}) }\leq e\))} \notag \\
    &\quad \geq P\left(N(t^{\frac{2}{3}}_n,x_{T_n}+V)\geq \delta f( t^{\frac{2}{3}}_n ,x_{T_n} +V )\right) > \epsilon. && \qquad \text{by \footnotesize \eqref{jio4}} 
\end{align}

\begin{figure}[ht]
    \centering
    \includegraphics[width=1\textwidth]{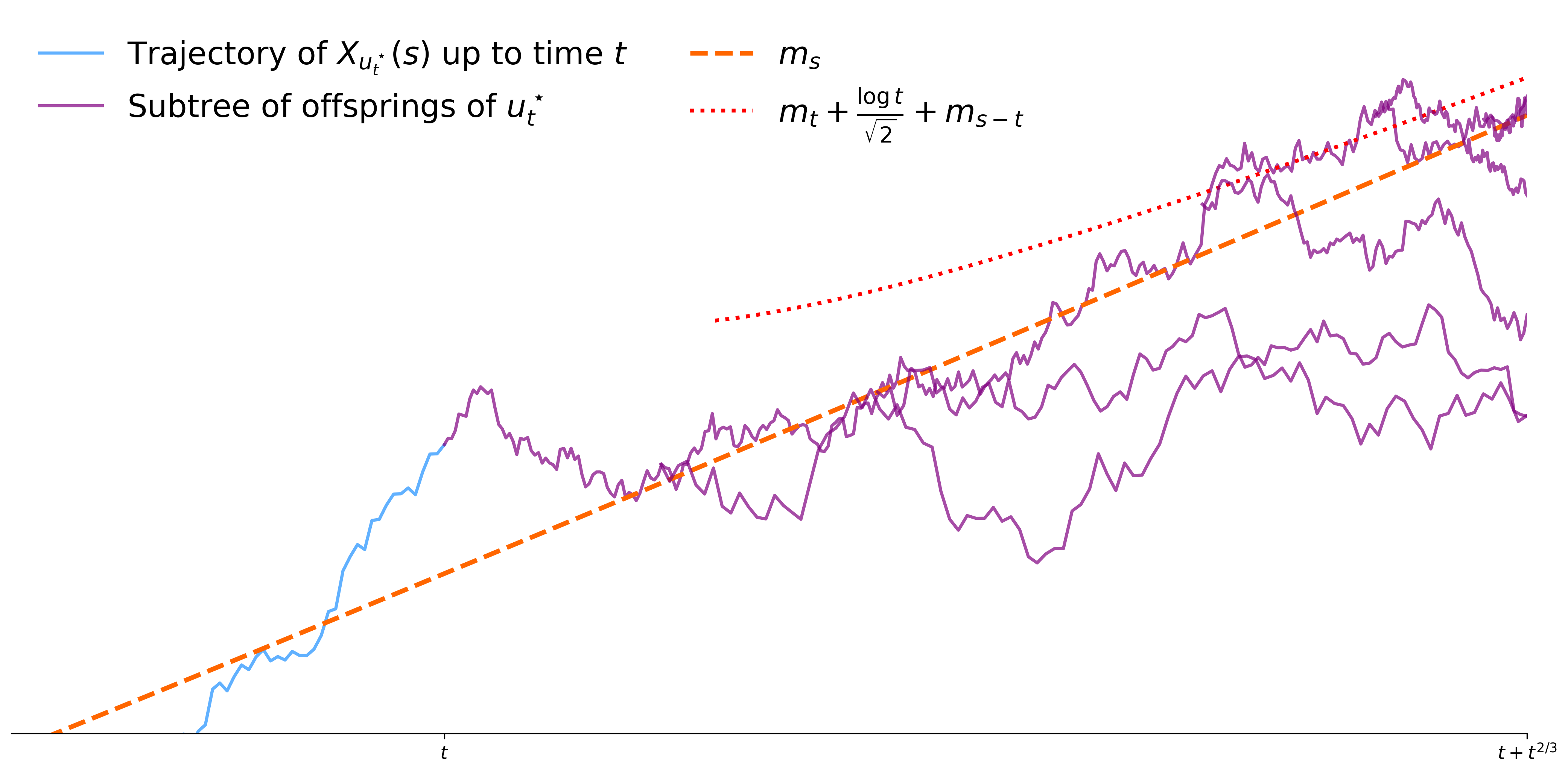}
    \caption{Illustration of a rare event—which occurs infinitely often almost surely—where \( M_t \) exceeds \( m_t + \frac{\log(t)}{\sqrt{2}} \), resulting in an inflation of \( N(s, x_s) \) for \( x_s \leq s^{1/3} \) up to time \( t + t^{2/3} \).}
    \label{fig:nocvps}
\end{figure}
    
\end{proof}

\section{Brownian estimates and toolbox}\label{sec:toolbox}

This section collects several useful results, primarily from Bramson \cite{Bramson1983ConvergenceOS} and Arguin et al. \cite{Arguin2016}.

Define the probability measure $\Prob_{0,x}^{t,y}$ under which $(B_s)_{s \in [0,t]}$ is a Brownian bridge from $x$ at time $0$ to $y$ at time $t$. For a function $l : [s_1,s_2] \to \mathbb{R}$, we denote by $B_l$ (or $B_l[s_1,s_2]$ in cases of ambiguity) the event that the path lies strictly above $l$ on the interval $[s_1, s_2]$. Similarly, $B^l$ denotes the event that the path lies strictly below $l$.

The following is a well-known expression for the probability that a Brownian bridge stays non-negative.
\begin{lemma}\label{bmbar}
\begin{equation}
    \Prob_{0,x}^{t,y}(B_s \geq 0 , s \in [0,t]) = 1 - e^{-\frac{2xy}{t}}.
\end{equation}
\end{lemma}

The following result provides an upper bound on the probability that a Brownian bridge stays below a linear function.
\begin{lemma}[Lemma 3.4, \cite{Arguin2016}]\label{34Arguin}
Let \( Z_1, Z_2 \geq 0 \) and \( r_1, r_2 \geq 0 \). Then, for \( t > r_1 + r_2 \),
\[
\mathbb{P}_{0,0}^{t,0} \left( B_s \leq \left(1 - \frac{s}{t} \right) Z_1 + \frac{s}{t} Z_2, \quad \forall s \in [r_1, t - r_2] \right) \leq \frac{2}{t - r_1 - r_2} \prod_{i=1}^2 \left( Z(r_i) + \sqrt{r_i} \right),
\]
where
\[
Z(r_1) \coloneqq \left(1 - \frac{r_1}{t} \right) Z_1 + \frac{r_1}{t} Z_2, \quad Z(r_2) \coloneqq \frac{r_2}{t} Z_1 + \left(1 - \frac{r_2}{t} \right) Z_2.
\]
\end{lemma}

The following identity, derived from the reflection principle, gives the probability that a Brownian bridge remains non-negative over an interval.

\begin{lemma}\label{reflectionprinciple0}
For $y>0$ and $0<r<\gamma$,
\begin{equation}
    \Prob(\forall s \in [r, \gamma] \, , \, B_s \geq 0 \mid B_\gamma = y ) = 1- 2\Prob(B_r \leq 0 \mid B_\gamma = y ).
\end{equation}
\end{lemma}
\begin{proof}
    We prove that 
    \begin{equation}
        \Prob(\exists s \in [r, \gamma] \, , \, B_s < 0 \mid B_\gamma = y ) =2\Prob_{0,0}^{\gamma,y}(B_r<0).
    \end{equation}
    Indeed,
    \begin{equation}
        \Prob_{0,0}^{\gamma,y}(\exists s \in [r, \gamma] \, , \, B_s < 0)= \Prob_{0,0}^{\gamma,y}(B_r<0) + \Prob_{0,0}^{\gamma,y}(B_r>0 \, \cap \, \exists s \in [r, \gamma] \, , \, B_s < 0 ).
    \end{equation}
    By the reflection principle,
    \begin{equation}
        \Prob_{0,0}^{\gamma,y}(B_r>0 \, \cap \, \exists s \in [r, \gamma] \, , \, B_s < 0 )=\Prob_{0,0}^{\gamma,-y}(B_r>0 )=\Prob_{0,0}^{\gamma,y}(B_r<0 ).
    \end{equation}
    Hence the result.
\end{proof}

We now state a useful monotonicity result.
\begin{lemma}[Lemma 2.6, \cite{Bramson1983ConvergenceOS}]\label{monoticitypath}
    Assume that $l_1(s)$, $l_2(s)$, and $\wedge(s)$ satisfy $l_1(s)\leq l_2(s) \leq \wedge(s)$ for $s\in [0,t]$, and that $\Prob(B_{l_2}[0,t])>0$. Then,
    \begin{equation}
        \Prob_{0,0}^{t,0}(B^{\wedge}\mid B_{l_1})\geq \Prob_{0,0}^{t,0}(B^{\wedge}\mid B_{l_2}),
    \end{equation}
    and
    \begin{equation}
        \Prob_{0,0}^{t,0}(B_{\wedge}\mid B_{l_1})\leq \Prob_{0,0}^{t,0}(B_{\wedge}\mid B_{l_2}).
    \end{equation}
\end{lemma}

The following result establishes that if a Brownian bridge remains above \(-\Gamma_{t,\alpha}\) for some \(\alpha \in (0,1/2)\) on the interval \([r, t - r]\), then it is remains above \(\Gamma_{t,\alpha}\) on the same interval.
\begin{lemma}[Proposition 6.1, \cite{Bramson1983ConvergenceOS}]\label{prop61bramson}
    Let $C>0$ and for all $t$, $L_t : [0,t]\mapsto \mathbb{R}$ be a family of functions.

Let any $\alpha \in ]0,1/2[$, if there exists $r_0$ such that $L_t(s) -C\Gamma_{t,\alpha}(s)) \leq 0$ for $s \in [r_0,t-r_0]$ for all $t$, then,
\begin{equation}
        \lim_{r\to \infty}\sup_{t: t\geq 3r}\frac{\Prob_{0,0}^{t,0}( \forall s \in [r,t-r] : B_s > L_t(s) +C\Gamma_{t,\alpha}(s))}{\Prob_{0,0}^{t,0}( \forall s \in [r,t-r] : B_s > L_t(s) -C\Gamma_{t,\alpha}(s))} =0.
    \end{equation}
\end{lemma}

The following result, which is proved via a coupling argument, shows that the ratio considered in Lemma \ref{prop61bramson} is increasing with respect to the endpoint.
\begin{lemma}[Proposition 6.3, \cite{Bramson1983ConvergenceOS}]\label{prop63bramson} Assume that the functions $l_1$ and $l_2$ are upper semi-continuous except at at most finitely many points, and that $l_1(s)\leq l_2(s)$ for $s\in [0,t]$. Then, provided that the denominator is nonzero,
\begin{equation}
      \begin{aligned}
               & \frac{\Prob_{0,0}^{t,y}( \forall s \in [r,t-r] : B_s > l_2(s))}{\Prob_{0,0}^{t,y}( \forall s \in [r,t-r] : B_s > l_1(s))}     & & \text{ is increasing in $y$.}\\
      \end{aligned}
    \end{equation}
\end{lemma}

The following are standard results (see, e.g., \cite{harris2014manytofew}) in the theory of branching processes, which enable one to compute the first and second moments by reducing the problem to a single-particle or a two-particle system, respectively.

\begin{lemma}[Many-to-one Lemma]
    For any $t\geq 0$ and any measurable function $f: C([0,t])  \rightarrow \mathbb{R^+}$,
\begin{equation*}
    \E\left[\sum_{u\in \mathcal{N}_t}f(X_u(s),s \in \left[0,t\right])\right]= e^t \E\left[ f(B_s, s\in \left[0,t\right] )\right],
\end{equation*}
where $(B_s)_{s\in \mathbb{R}}$ is a standard Brownian motion.
\end{lemma}

\begin{lemma}[Many-to-two Lemma]\label{manyto2}
    For any $t\geq 0$ and any measurable function $f: C([0,t])  \rightarrow \mathbb{R^+}$,
\begin{equation}
    \begin{split}
        &\E\left[ \left(\sum_{u\in \mathcal{N}_t}f(X_u(r),r \in \left[0,t\right])\right)^2 \right] = e^t \E\left[ f(B_r, r\in \left[0,t\right] )^2\right]\\
        & \qquad  + \int_0^t 2e^{2t-s}\E\left[ f(B^{(1,s)}_r, r\in \left[0,t\right])f(B^{(2,s)}_r, r\in \left[0,t\right])  \right]\mathrm{d}s,
    \end{split}
\end{equation}
where for $s \in \left[0,t\right]$, $(B^{(1,s)}_r)_{r\in \mathbb{R}}$ and $(B^{(2,s)}_r)_{r\in \mathbb{R}}$ follow a common Brownian motion trajectory until time $s$, then split and evolve as independent Brownian motions afterwards.
\end{lemma}

\section{Proof of intermediary Propositions}\label{Sec:proofprop}

In this section, we provide the proof of the results stated in Section \ref{sec:proofff}. We denote by \( p_t(x) = (2\pi t)^{-1/2} e^{-x^2/(2t)} \) the Gaussian density with mean \(0\) and variance \(t\).

\subsection{Localisation of paths}\label{subSec:proofprop1}
Here, we prove Propositions \ref{prop:Kdist} and \ref{pathlocfin}. To do so, we begin by establishing a first path localisation for particles contributing to $N(t,x)$. Specifically, we uniformly bound the number of particles in $\mathcal{N}(t,x)$ that cross the barrier $\beta_t - \Gamma_{t,\alpha}$ during $ [r, t-r]$ for large $r$:

\begin{proposition}\label{pathlocup_proba}
    Let $\alpha\in [0,1/2)$, for any $\delta>0$,
\begin{equation}
    \lim_{r\to \infty}\sup_{x,t: t>3r,  x\leq t}\Prob\left( \frac{\NOB{\beta_t -\Gamma_{t,\alpha}}{t,[r,t-r],x}}{f(t,x)}  >\delta \right) =0.
\end{equation}
\end{proposition}

The proof of Proposition \ref{pathlocup_proba} is inspired by the strategy developed in \cite{Arguin2016}. We will use one of their key results, which provides an upper bound for the trajectory of particles in a BBM.
\begin{theorem}[Theorem 2.2, \cite{Arguin2016}]\label{envsuparguin}
    Let $0 < \gamma < 1/2$. Let also $y \in \mathbb{R}, \epsilon > 0$ be given. There
exists $r_u = r_u(\gamma, y, \epsilon)$ such that for $r \geq r_u$ and for any $t > 3r$,
\begin{equation}
    \Prob\left(\exists u \in \mathcal{N}_t: X_u(s) > y + s\frac{m_t}{t}+\Gamma_{t,\gamma}(s), \text{ for some } s \in [r, t - r]\right)< \epsilon
\end{equation}
\end{theorem}

This result is crucial in the study of the localisation of particles in $N(t,x)$. Indeed, the fact that no particles exceed $ s\frac{m_t}{t}+\Gamma_{t,\gamma}(s)$ is sufficient to understand the typical displacement $m_t$. We can then apply the same `entropic repulsion' idea as in \cite{Arguin2016} to get a control on the path of particles at sublinear distances of $m_t$.

\begin{proof}[Proof of Proposition \ref{pathlocup_proba}] Proposition \ref{pathlocup_proba} states that for any $0 \leq \alpha < 1/2$, 
\begin{equation}
    \frac{\#\{u \in \mathcal{N}(t,x): X_u(s) >  s\frac{m_t}{t} - \Gamma_{t,\alpha}(s), \text{ for some } s \in [r, t - r]\}}{f(t,x)} \xrightarrow[r\rightarrow \infty]{\mathbb{P}}  0
\end{equation}
uniformly in $t>3r$ and $t \geq x$.

Let $\epsilon>0$. Recall that $M_t-m_t$ is tight, and so there exists $K>0$ and $t_0$ such that $\Prob(M_t -m_t >K)<\epsilon$ for $t\geq t_0$ large enough. Denote the event $E_{K,t}\coloneqq \{M_t -m_t \leq K\}$.

Denote $I_r=[r,t-r]$ and $I_R=[R,t-R]$. By Theorem \ref{envsuparguin} (with $y=0$), let $r_0$ such that for $r \geq r_0$ and for any $t > 3r$,
\begin{equation}
    \Prob\left(\exists u \in \mathcal{N}_t: X_u(s) > s\frac{m_t}{t}+\Gamma_{t,\alpha}(s), \text{ for some } s \in [r, t - r]\right)< \epsilon.
\end{equation}
By the many-to-one Lemma,
    \begin{equation}\label{gfdgd}
    \begin{split}
         \E&\left[\bbone_{E_{K,t}} \# \mathcal{N}^{\beta_t+ \Gamma_{t,\alpha}}(t, I_r,x) \setminus {\mathcal{H}^{\beta_t - \Gamma_{t,\alpha}}(t, I_R)} \right]\\& \leq e^t  \int_{-K}^{x} \phi_t(m_t -z) 
           \Prob_{0,0}^{t,z}(\exists s \in I_R  :  B_s <  \Gamma_{t,\alpha}  \text{ and } \forall s\in I_r: B_s \geq - \Gamma_{t,\alpha})\mathop{}\!\mathrm{d}z.
    \end{split}
    \end{equation}
    For $z\in [-K,x]$ we now bound 
    \begin{align}\label{2termadbound}
       & \Prob_{0,0}^{t,z}(\exists s \in I_R \, , \, B_s <  \Gamma_{t,\alpha}   \cap \forall s\in I_r, B_s \geq - \Gamma_{t,\alpha})\\
        &\qquad =\Prob_{0,0}^{t,z}( \forall s\in I_r, B_s \geq - \Gamma_{t,\alpha})\Prob_{0,0}^{t,z}(\exists s \in I_R \, , \, B_s <  \Gamma_{t,\alpha} \mid \forall s\in I_r, B_s \geq - \Gamma_{t,\alpha}).\notag 
    \end{align}
    We deal with each of the two factors in turn. First,
    \begin{equation}
    \begin{split}
        \Prob_{0,0}^{t,z} &\left( \forall s \in I_r,\ B_s \geq -\Gamma_{t,\alpha} \right) 
        \\ &= \Prob_{0,0}^{t,z} \left( \forall s \in I_r,\ B_s \geq 0 \right)
         \Prob_{0,0}^{t,z} \left( \forall s \in I_r,\ B_s \geq 0 \mid \forall s \in I_r,\ B_s \geq -\Gamma_{t,\alpha} \right)^{-1}
        \\& \leq 2 \frac{\left( \sqrt{r} + z \right)\left( \sqrt{r} + r \frac{z}{t} \right)}{t - 2r} \Prob_{0,0}^{t,z} \left( \forall s \in I_r,\ B_s \geq \Gamma_{t,\alpha} \mid \forall s \in I_r,\ B_s \geq -\Gamma_{t,\alpha} \right)^{-1}
        \\& \leq 3 \frac{\left( \sqrt{r} + z \right)(r + 1)}{t - 2r} \quad \text{for } r \text{ large enough.}
    \end{split}   
\end{equation}

    The penultimate inequality follows from Lemma \ref{34Arguin} for the first term. The last inequality is a consequence Lemma \ref{prop63bramson} and  Lemma \ref{prop61bramson}.  
    
    We now bound the second term in \eqref{2termadbound}:
    \begin{equation}
      \begin{aligned}
        \Prob_{0,0}^{t,z}&(\exists s \in I_R \, , \, B_s < \Gamma_{t,\alpha}  \mid \forall s\in I_r, B_s \geq - \Gamma_{t,\alpha}) & &  \\
               & \leq \Prob_{0,0}^{t,z}(\exists s \in I_R \, , \, B_s <  \Gamma_{t,\alpha} \mid \forall s\in I_R, B_s \geq - \Gamma_{t,\alpha}) & & \text{\footnotesize(Lemma \ref{monoticitypath})}\\
                & = 1-\Prob_{0,0}^{t,z}(\forall s \in I_R \, , \, B_s \geq  \Gamma_{t,\alpha} \mid \forall s\in I_R, B_s \geq - \Gamma_{t,\alpha}) & & \\
                 & \leq 1-\Prob_{0,0}^{t,-K}(\forall s \in I_R \, , \, B_s \geq  \Gamma_{t,\alpha} \mid \forall s\in I_R, B_s \geq - \Gamma_{t,\alpha}) & & \text{\footnotesize(Lemma \ref{prop63bramson})}\\
                 & = o_R(1). & & \text{\footnotesize(Lemma \ref{prop61bramson})}
      \end{aligned}
    \end{equation}
    
    Therefore, for $r$ large enough,  \eqref{2termadbound} is smaller than
    \begin{equation}
        3\frac{\left(\sqrt{r}+z\right)(r+1)}{t-2r}o_R(1).
    \end{equation}
    Thus, for $t\geq 3r$, equation \eqref{gfdgd} is smaller than
    \begin{equation}\label{48}
        \begin{split}
           3o_R(1)(r+1)\int_{-K}^x  (\sqrt{r}+y)e^{\frac{m_t}{t} y -\frac{y^2}{2t}} \mathop{}\!\mathrm{d}y.
    \end{split}
    \end{equation}

As $x\leq t$ and $t\geq 3r$, for $r$ large enough, by the Laplace method
\[
\int_{-K}^x  (\sqrt{r}+z)e^{\sqrt{2}z - \frac{3z\log t}{2\sqrt{2}t} -\frac{z^2}{2t}} \mathop{}\!\mathrm{d}z \le C (\sqrt{r}+x) e^{\sqrt{2}x - \frac{3x\log t}{2\sqrt{2}t} -\frac{x^2}{2t}}.
\]
Thus, \eqref{48} is bounded by
\begin{equation}
         o_R(1)(r+1)^{\frac{3}{2}} x e^{\sqrt{2}x - \frac{3x\log t}{2\sqrt{2}t} -\frac{x^2}{2t}} \leq C^{*}o_R(1) f(t,x)(r+1)^{\frac{3}{2}}.
\end{equation}
This yields that, for any $\delta$, for $r\geq \max(r_0, \frac{t_0}{3})$, $R\geq r$ and $x,t$ such that $t>3r$ and $ x\leq t$, by Markov's inequality,
    \begin{equation}
        \begin{split}
        \Prob&\left( \NOB{\beta_t -\Gamma_{t,\alpha}}{t,I_R,x}  >\delta f(t,x)\right)
        \\&\leq \Prob\left(E_{K,t}^c\right) + \Prob\left(\exists u \in \mathcal{N}_t: X_u(s) > s\frac{m_t}{t}+\Gamma_{t,\alpha}(s), \text{ for some } s \in [r, t - r]\right)
        \\& \qquad +P \left(\bbone_{E_K} \# \mathcal{N}^{\beta_t+ \Gamma_{t,\alpha}}(t, I_r,x) \setminus {\mathcal{H}^{\beta_t - \Gamma_{t,\alpha}}(t, I_R)} >\delta f(t,x) \right)
        \\& \leq 2\epsilon + (\delta f(t,x))^{-1}\E\left[\bbone_{E_K} \# \mathcal{N}^{\beta_t+ \Gamma_{t,\alpha}}(t,I_r,x) \setminus {\mathcal{H}^{\beta_t - \Gamma_{t,\alpha}}(t, I_R)} \right]
        \\& \leq 2\epsilon +\delta^{-1}3C^{*} o_R(1) (r+1)^{\frac{3}{2}}.
        \end{split}   
    \end{equation}
Hence, taking $R$ large enough, 
\begin{equation}
   \sup_{x,t: t>3r_0, x\leq t}\Prob( \NOB{\beta_t -\Gamma_{t,\alpha}}{t,I_R,x}  >\delta f(t,x)) \leq 3\epsilon.
\end{equation}
Thus for any $\alpha\in [0,1/2)$ and any $\delta>0$,
\begin{equation}
   \lim_{R\to \infty}\sup_{x,t: t>3R, x\leq t}\Prob( \NOB{\beta_t -\Gamma_{t,\alpha}}{t,I_R,x}  >\delta f(t,x)) =0.
\end{equation}

\end{proof}

We now prove Proposition \ref{prop:Kdist}.

\begin{proof}[Proof of Proposition \ref{prop:Kdist}] Proposition \ref{prop:Kdist} states that for any $\delta>0$,
\begin{equation}\label{kih}
    \lim_{K\to\infty} \lim_{x\to\infty} \sup_{t\geq x}\Prob\left(\frac{N(t,x-K)}{f(t,x)} > \delta\right) = 0.
\end{equation}

Let $I=[r,t-r]$, since
\begin{equation}
    N(t,x-K)=N^{\beta_t}(t,I,x-K) + \NOB{\beta_t}{t,I,x-K} ,
\end{equation}
the claim \eqref{kih} will be a consequence of the two following results:
For any $\delta>0$, 
\begin{equation}\label{Kr1}
    \lim_{r\to \infty}\lim_{K\to\infty} \lim_{x\to\infty} \sup_{t\geq x}\Prob\left(\frac{N^{\beta_t}(t,I,x-K)}{f(t,x)} > \delta\right) = 0,
\end{equation}
and 
\begin{equation}\label{Kr2}
    \lim_{r\to \infty}\lim_{K\to\infty} \lim_{x\to\infty} \sup_{t\geq x}\Prob\left(\frac{\NOB{\beta_t}{t,I,x-K}}{f(t,x)} > \delta\right) = 0.
\end{equation}
We first prove \eqref{Kr2}, which is a direct consequence of Proposition \ref{pathlocup_proba}. Indeed, we have that,
\begin{equation}
     \NOB{\beta_t}{t,I,x-K}\leq  \NOB{\beta_t}{t,I,x}.
\end{equation}
Moreover, Proposition \ref{pathlocup_proba} implies that 
\begin{equation}
    \lim_{r\to \infty}\lim_{x\to\infty}\sup_{t\geq x}\Prob\left(\frac{\NOB{\beta_t}{t,I,x}}{f(t,x)} > \delta\right) = 0.
\end{equation}
Thus \eqref{Kr2} holds.

Let us now show \eqref{Kr1}.

By the many-to-one Lemma and Lemma \ref{34Arguin}, 
    \begin{equation}\label{gfdgd93}
    \begin{split}
         \E&\left[ N^{\beta_t}(t,I,x-K) \right]=e^t  \int_{-\infty}^{x -K} \phi_t(m_t -z) 
           \Prob_{0,0}^{t,z}(\forall s\in I, B_s \geq 0)\mathop{}\!\mathrm{d}z
           \\& \leq  \Prob_{0,0}^{t,x}(\forall s\in I, B_s \geq 0)e^t \int_{-\infty}^{x -K} \phi_t(m_t -z) \mathop{}\!\mathrm{d}z
           \\& \leq \left(\frac{t}{t-2r}\right)\left(\sqrt{r}+x\frac{t-r}{t}\right)\left(\sqrt{r}+r\frac{x}{t}\right)\int_{-\infty}^{x -K} e^{\frac{m_t}{t}y -\frac{y^2}{2t}}\mathop{}\!\mathrm{d}y
           \\& \leq \left(\frac{t}{t-2r}\right)(\sqrt{r}+x)(r+1)\int_{-\infty}^{x -K} e^{\frac{m_t}{t}y -\frac{y^2}{2t}}\mathop{}\!\mathrm{d}y.
    \end{split}
    \end{equation}

As $x\leq t$ and $\lim_{t \to \infty} \frac{m_t}{t} = \sqrt{2}$, for $t$ large enough we have $\frac{m_t}{t} - \frac{x}{t} \geq c > 0$ for some constant $c \in (0, \sqrt{2}-1)$. Thus, this is bounded by
\begin{equation}
         Cf(t,x-K)\frac{t}{t-2r}(r+1)^{\frac{3}{2}} \leq Cf(t,x) e^{-cK}\frac{t}{t-2r}(r+1)^{\frac{3}{2}} .
\end{equation}

Hence,
\begin{equation}
    \lim_{r\to \infty}\lim_{K\to\infty} \lim_{x\to\infty} \sup_{t: t\geq x}\frac{\E\left[ N^{\beta_t}(t,I,x-K) \right]}{f(t,x)} = 0.
\end{equation}
Thus  \eqref{Kr1} holds by Markov's inequality.
\end{proof}

We now prove the last path localisation Proposition \ref{pathlocfin}.

\begin{proof}[Proof of Proposition \ref{pathlocfin}]
Proposition \ref{pathlocfin} states that for $\alpha\in [0,1/2[$, $\lambda>\frac{1}{2}$, $\rho_{t,r}(s)= \beta_t(s)-\bbone_{s\in [r,t-r]}\Gamma_{t,\alpha}(s)$ and $J=[r,t]$. For any $\delta>0$,
\begin{equation}
    \lim_{r\to \infty}\sup_{t,x: t>3r,  r^{\lambda}\leq x\leq t}\Prob\left( \frac{\NOB{\rho_{t,r}}{t, J,x}}{f(t,x)} >\delta\right) =0.
\end{equation}
For any $K>0$,
\begin{equation}\label{fdgf3}
   \NOB{\rho_{t,r}}{t, J,x} \leq N(t,x-K)+\NOB{\rho_{t,r}}{t, J,[x-K,x] } .
\end{equation}
First note that Proposition \ref{prop:Kdist} implies that, for any $\delta>0$
\begin{equation}
    \lim_{K\to \infty}\lim_{r\to \infty}\sup_{t>3r,  r^{\lambda}\leq x\leq t}\Prob\left( \frac{N(t,x-K)}{f(t,x)} >\delta\right) =0.
\end{equation}
We now turn to the second term of the LHS of equation \eqref{fdgf3}. Note that:
\begin{equation}\label{gfdhr4}
    \begin{split}
    \NOB{\rho_{t,r}}{t, J,[x-K,x] }&=\NOB{\beta_t-\Gamma_{t,\alpha}}{t, I_r,[x-K,x] }\\
    &+\#  \nB{\beta_t-\Gamma_{t,\alpha}}{t,I_r,[x-K,x]} \setminus \mathcal{H}^{\beta_t}(t, [t-r,t]).
    \end{split}
\end{equation}
Moreover by Proposition \ref{pathlocup_proba}, for any $\delta>0$,
\begin{equation}
    \lim_{r\to \infty}\sup_{t>3r,  r^{\lambda}\leq x\leq t}\Prob\left( \frac{\NOB{\beta_t-\Gamma_{t,\alpha}}{t, I_r,[x-K,x] } }{f(t,x)} >\delta\right) =0.
\end{equation}

It thus remains to bound the last term in equation \eqref{gfdhr4}.
Let $\lambda>\frac{1}{2}$, $\alpha \in ]1-\lambda, \frac{1}{2}[$, and $x\geq r^\lambda$. By the many-to-one Lemma, 
\begin{equation}\label{gfdtg}
    \begin{split}
         \E&\left[  \#  \nB{\beta_t-\Gamma_{t,\alpha}}{t,I_r,[x-K,x]} \setminus \mathcal{H}^{\beta_t}(t, [t-r,t])   \right]\\
         & =  e^t \Prob\left( B_t -m_t \in [-x,-x+K],\ \forall s\in [r,t-r]: B_s \leq \beta_t(s)-\Gamma_{t,\alpha}(s), \right. \\
         & \qquad\qquad\left. \exists s \in [t-r,t] : B_s\geq s\frac{m_t}{t} \right)
         \\& = e^t  \int_{r^\alpha}^{+\infty} \phi_{t-r}\left((t-r)\frac{m_t}{t} -y\right)\Prob_{0,0}^{t-r,-y}\left(\forall s\in [r,t-r]: B_s \leq -\Gamma_{t,\alpha}(s)\right) 
         \\& \qquad \Prob\left( B_t -m_t \in [-x,-x+K],\ \exists s \in [t-r,t] : B_s\geq s\frac{m_t}{t} \mid B_{t-r}=\tfrac{t-r}{r}m_t -y \right)\mathop{}\!\mathrm{d}y
         \\& \leq e^t  \int_{r^\alpha}^{+\infty} \phi_{t-r}\left((t-r)\frac{m_t}{t} -y\right)\Prob_{0,0}^{t-r,-y}\left(\forall s\in [r,t-r]: B_s \leq 0\right) 
         \\& \qquad \int_{x-K}^x \phi_{r}\left(r\frac{m_t}{t} +y-z\right)\Prob_{t-r,y}^{t,z}\left( \exists s \in [t-r,t] : B_s\leq 0 \right)\mathop{}\!\mathrm{d}z \mathop{}\!\mathrm{d}y.
    \end{split}
\end{equation}

By Lemma 3.4 in \cite{arguin2011extremal}, and for $t\geq 3r$,
\begin{equation}
    \Prob_{0,0}^{t-r,-y} \left( \forall s \in [r, t - r] : B_s \leq 0 \right) \leq  \frac{2y \left( \sqrt{r} + \frac{r}{t - r} y \right)}{t - 2r} \leq \frac{6}{t} \, y^2 \left( \sqrt{r} + 1 \right),
\end{equation}

and by Lemma \ref{bmbar},
\begin{equation}
    \Prob_{t-r,y}^{t,z} \left( \exists s \in [t - r, t] : B_s \leq 0 \right) = e^{-2 \,\frac{zy}{r}}.
\end{equation}

Thus, for $t\geq 3r$, equation \eqref{gfdtg} is smaller than

\begin{equation}\label{continuinggdfg}
    \begin{split}
        &  3\sqrt{2} \, e^{t - \frac{m_t^2}{2t}} \frac{1}{\pi t^{\frac{3}{2}} \sqrt{r}} \int_{r^\alpha}^{+\infty} y^2 \left( \sqrt{r} + 1 \right) \int_{x-K}^x e^{\frac{m_t}{t} \, z - \frac{(z - y)^2}{2r}} e^{-2 \,\frac{zy}{r}} \mathop{}\!\mathrm{d}z \mathop{}\!\mathrm{d}y
         \\&\leq C \, \frac{\sqrt{r} + 1}{\sqrt{r}} \int_{r^\alpha}^{+\infty} y^2 \int_{x-K}^x e^{-\frac{(z - y)^2}{2r}} e^{-2 \,\frac{zy}{r}} e^{\frac{m_t}{t} \, z} \mathop{}\!\mathrm{d}z \mathop{}\!\mathrm{d}y
         \\&= C \, \frac{r + 1}{\sqrt{r}} \int_{r^\alpha}^{+\infty} y^2 \int_{x-K}^x e^{-\frac{(z + y)^2}{2r}} e^{\frac{m_t}{t} \, z} \mathop{}\!\mathrm{d}z \mathop{}\!\mathrm{d}y
         \\&\leq C \, \frac{r + 1}{\sqrt{r}} \int_{r^\alpha}^{+\infty} y^2 e^{-\frac{(x + y - K)^2}{2r}} \int_{x-K}^x e^{\frac{m_t}{t} \, z} \mathop{}\!\mathrm{d}z \mathop{}\!\mathrm{d}y
         \\&\leq C \, \frac{r + 1}{\sqrt{r}} e^{\frac{m_t}{t} \, x} \int_{r^\alpha}^{+\infty} y^2 e^{-\frac{(x + y - K)^2}{2r}} \mathop{}\!\mathrm{d}y
         \\&= C \, \frac{r + 1}{\sqrt{r}} e^{\frac{m_t}{t} \, x} \int_{x + r^\alpha - K}^{+\infty} y^2 e^{-\frac{y^2}{2r}} \mathop{}\!\mathrm{d}y.
    \end{split}
\end{equation}

Since $\lambda>\frac{1}{2}$, for $x\geq r^\lambda$, there exists a constant $C$ such that \eqref{continuinggdfg} is smaller than
\begin{equation}
    \begin{aligned}
        & C(r + 1) r^{1/2} \, xe^{\frac{m_t}{t} x} e^{-\frac{\left(x + r^\alpha - K\right)^2}{2r}}  
        && \\
        & \leq C(r + 1) r^{1/2} \, xe^{\frac{m_t}{t} x - \frac{x^2}{2r} + \frac{x}{r^{1 - \alpha}}} 
        && \\
        & \leq xe^{\frac{m_t}{t} x - \frac{x^2}{2t}} \, C(r + 1) r^{1/2} \, e^{ -\frac{x^2}{3r} + \frac{x}{r^{1 - \alpha}}} 
        && \text{\scriptsize ($t \geq 3r$)} \\
        & = o_r(1) \, f(t, x). 
        && \text{\scriptsize ($x \geq r^\lambda$)}
    \end{aligned}
\end{equation}

Thus by Markov's inequality, we have that for any $\delta>0$,
\begin{equation}
    \lim_{r\to \infty}\sup_{t,x: t>3r,  r^{\lambda}\leq x\leq t}\Prob\left( \frac{\#  \nB{\beta_t-\Gamma_{t,\alpha}}{t,I_r,[x-K,x]} \setminus \mathcal{H}^{\beta_t}(t, [t-r,t]) }{f(t,x)} >\delta\right) =0.
\end{equation}
\end{proof}

\subsection{Conditional expectation}\label{subSec:proofprop2}

In this section, we prove Proposition \ref{espcondr}.

\begin{proof}[Proof of Proposition \ref{espcondr}]Let $x=x_t$, $K$, $r$, $R$ be such that, as $t\to \infty$, $x_t=o(t)$ and $x_t \to \infty$,  $K\to\infty$, $r \to \infty$, $R=o(\sqrt{t})$, $R=o(t/x_t)$, and $R\geq r$.
For $u\in \mathcal{N}_R$, we denote ${\mathcal{N}}^u_{t}= \{v\in \mathcal{N}_t : u\preceq v\}$ the set of descendant of $u$ at time $t$. 
Moreover,  define  $A_{v}\coloneqq\{ m_t- X_v(t) \in [x-K,x], \forall s \in [R,t]\,:  X_v(s)\leq  \beta_t(s)\}$ for $v\in \mathcal{N}_{t}$.

Applying the strong Markov property of branching Brownian motion at time $R$ together with the many-to-one lemma, we obtain:
\begin{equation}\label{fd3}
    \begin{split}
        & \E\left[ \NB{\beta_t}{t, [r,t],[x-K,x]}\mid \mathcal{F}_{R} \right] = \E\left[ \sum_{u\in \mathcal{H}^{\beta_t}(t,[r,t])}\bbone_{X_u(t)\geq m_t -x} \mid \mathcal{F}_{R} \right]\\
        &=\E\left[ \sum_{p\in \mathcal{H}^{\beta_t}(R,[r,R])}\sum_{q\in \mathcal{N}^p_{t}} \bbone_{A_q} \mid \mathcal{F}_{R} \right]=\sum_{p\in \mathcal{H}^{\beta_t}(R,[r,R])}\E\left[ \sum_{q\in \mathcal{N}^p_{t}} \bbone_{A_q} \mid \mathcal{F}_{R} \right]\\
        &=\sum_{p\in \mathcal{H}^{\beta_t}(R,[r,R])}e^{t-R} \Prob_{R,X_p(R)}\left( m_t -B_t \in [x-K,x] \cap \forall s\in [R,t] :
          B_s\leq s\frac{m_t}{t} \right).\\
    \end{split}
\end{equation}

For each $p\in \mathcal{N}_{R}$, denote $y_p= R\frac{m_t}{t} - X_p(R)$. We have that,
\begin{equation}\label{fgfrfg}
    \begin{split}
        \Prob_{R,X_p(R)}&\left( m_t -x\leq B_t\leq m_t -x +K  \cap \forall s \in [R,t] :
          B_s\leq s\frac{m_t}{t} \right)\\
        &= \int_{m_t - x}^{m_t-x+K} \Prob_{R,X_p(R)}^{t,z}(\forall s \in [R,t] : B_s \leq s\frac{m_t}{t}) \Prob_{R,X_p(R)}(B_t \in \mathop{}\!\mathrm{d}z) \\
        &= \int_{x-K}^x \Prob_{0,y_p}^{t-R,z}(\forall s\in [0,t-R]:B_s \geq 0) \Prob_{R,X_p(R)}(m_t - B_t \in \mathop{}\!\mathrm{d}z). 
    \end{split}
\end{equation}

On one hand, as $\lvert X_p(R) \rvert \leq \sqrt{2}R$ uniformly in $p$, eventually almost surely and since $R=o(t/x)$, $\frac{y_p x}{t-R}$ goes to $0$ uniformly in $p$. Consequently, for $z\in[0,x]$, 
\begin{equation}
    \Prob_{0,y_p}^{t-R,z}(\in [0,t-R]B_s \geq 0)= 1-e^{-2\frac{zy_p}{t-R}} = 2\frac{zy_p}{t-R} (1+o_t(1)),
\end{equation}
where, as $t\to \infty$, $o_t(1)$ goes to 0 uniformly in $p$ and $z\in[0,x]$, almost surely.

On the other hand,
\begin{equation*}
    \begin{split}
         &\sqrt{2\pi (t-R)}\Prob_{R,X_p(R)}(m_t- B_t \in \mathop{}\!\mathrm{d}z)\\
         &= e^{-\frac{\left(m_t -z -\left(R\frac{m_t}{t}-y_p\right)\right)^2}{2(t-R)}}\mathop{}\!\mathrm{d}z\\
        &= e^{-\frac{m_t^2(t-R)}{2t^2} +\frac{m_t}{t}(z-y_p)-\frac{(z-y_p)^2}{2(t-R)}}\mathop{}\!\mathrm{d}z\\
        &= e^{-(t-R)+\frac{3}{2}\log(t)(1-\frac{R}{t})-\frac{9\log(t)^2}{16t}(1-\frac{R}{t}) +\sqrt{2}(z-y_p) -\frac{3\log(t)}{2\sqrt{2}t}(z-y_p)- \frac{z^2}{2(t-R)} +\frac{zy_p}{t-R}-\frac{y_p^2}{2(t-R)}}\mathop{}\!\mathrm{d}z\\
        &= e^{-(t-R)+\frac{3}{2}\log(t) +\sqrt{2}(z-y_p)- \frac{z^2}{2t} - \frac{3z\log(t)}{2\sqrt{2}t} +o_t(1)}\mathop{}\!\mathrm{d}z,
    \end{split}
\end{equation*}
where, as $t\to \infty$, $o_t(1)$ goes to 0 almost surely, uniformly in $p$ and $z \in [0,x]$ because $R=o(\sqrt{t})$, $R=o(\frac{t}{x})$ and $x=o\left(\frac{t}{\log(t)}\right)$.
Substituting in \eqref{fgfrfg}, we obtain:
\begin{equation}
    \begin{split}
        \Prob_{R,X_p(R)}&\left( m_t -x\leq B_t\leq m_t -x +K  \cap \forall s \in [R,t]: B_s\leq s\frac{m_t}{t} \right)\\
        &=e^{-(t-R)}\left(\frac{t}{t-R}\right)^{\frac{3}{2}} \sqrt{\frac{2}{\pi}} y_pe^{-\sqrt{2}y_p}(1+o_t(1))\int_{x-K}^x ze^{\sqrt{2}z- \frac{z^2}{2t} - \frac{3z\log(t)}{2\sqrt{2}t}}\mathop{}\!\mathrm{d}z\\
        &= e^{-(t-R)} \sqrt{\frac{1}{\pi}} y_pe^{-\sqrt{2}y_p}xe^{\sqrt{2}x- \frac{x^2}{2t} - \frac{3x\log(t)}{2\sqrt{2}t}}(1+o_t(1))\\
        &=e^{-(t-R)}f(t,x)y_pe^{-\sqrt{2}y_p}(1+o_t(1)),
    \end{split}
\end{equation}

for some $o_t(1)$ converging to $0$ uniformly in $p$ almost surely.
The second equality comes from the fact that $R=o(t)$ and that $x=o(t)$.
Plugging this in \eqref{fd3}, for some $o_t(1)$ converging to $0$ almost surely, we conclude that

\begin{equation}
\begin{aligned}
    \E&\left[ \NB{\beta_t}{t, [r,t],[x-K,x]}\mid \mathcal{F}_{R}\right]
    \\&= f(t,x)\,(1+o_t(1)) \sum_{p\in \mathcal{N}_R} 
       \bbone_{ \forall s \in [r,R] : X_p(s)\leq \frac{m_t}{t}s }\, y_p\, e^{-\sqrt{2}y_p} \\
    &= f(t,x)\,(1+o_t(1)) \sum_{p\in \mathcal{N}_R} 
       \bbone_{ \forall s \in [r,R] : X_p(s)\leq \frac{m_t}{t}s  }\, y_p\, e^{-\sqrt{2}(\sqrt{2}R - X_p(R))} 
       &&  \text{\scriptsize ($R \log t / t \to 0$)} \\
    &= f(t,x)\,(1+o_t(1)) \sum_{p\in \mathcal{N}_R} 
       y_p\, e^{-\sqrt{2}(\sqrt{2}R - X_p(R))} 
       && \hspace{-2cm} \text{\scriptsize (eventually a.s.)} \\
    &= f(t,x)\, Z_R\, (1+o_t(1)) 
       && \hspace{-2cm} \text{\scriptsize (eventually a.s.)}
\end{aligned}
\end{equation}

To establish the penultimate equality, recall that \(\limsup_{t \to \infty} (M_t - \sqrt{2}t) = -\infty\) almost surely. Since \(r \leq R = o(\sqrt{t})\) and \(r \to \infty\) as \(t \to \infty\), the indicator function satisfies $\bbone_{\left\{ X_p(s) \leq \frac{m_t}{t} s \ \forall s \in [r, R] \right\}} = 1$, uniformly in \(p\), eventually almost surely. 

For the final equality, we expand $y_p = \sqrt{2}R - X_p(R) - \frac{3R\log t}{2\sqrt{2}t}$. Because $R = o(t/\log t)$, the exponential factor $e^{\frac{3R\log t}{2t}} = 1 + o_t(1)$. We also recall that the additive martingale $W_t\coloneqq\sum_{u\in \mathcal{N}_t}  e^{-\sqrt{2}\left(\sqrt{2}t - X_u(t)\right)}$ converges to zero almost surely as $t\to \infty$. Therefore, the residual term vanishes:
\begin{equation}
    \frac{3R\log t}{2\sqrt{2}t} \sum_{p\in \mathcal{N}_R}  e^{-\sqrt{2}\left(\sqrt{2}R - X_p(R)\right)} = \frac{3R\log t}{2\sqrt{2}t} W_R \xrightarrow[t\rightarrow \infty]{} 0 \quad \text{  a.s}.
\end{equation}

\end{proof}
\subsection{Second moment concentration}\label{subSec:proofprop3}

We first prove Lemma \ref{lemmaL2bound} 

\begin{proof}[Proof of Lemma \ref{lemmaL2bound}]
Take $ 0\leq R\leq t, X \subset \mathbb{R},I \subset [0,t]$ and a function $\beta \colon I \longrightarrow \mathbb{R}$.


For $u\in \mathcal{N}_R$, we denote ${\mathcal{N}}^u_{t}= \{v\in \mathcal{N}_t : u\preceq v\}$ the set of descendant of $u$ at time $t$. Further define  $A_{v}\coloneqq\{ m_t- X_v(t) \in X, \forall s \in [R,t]\cap I \,:  X_v(s)\leq  \beta(s)\}$ for $v\in \mathcal{N}_{t}$.



We have that
\begin{equation}
    \begin{split}
         \E&\left[ \left(N^{\beta}(t, I,X)- \E\left[N^{\beta}(t, I,X)\mid \mathcal{F}_R \right] \right)^2 \right]\\&
         =\E\left[N^{\beta}(t, I,X)^2 \right] + \E\left[ \E\left[N^{\beta}(t, I,X) \mid\mathcal{F}_R \right] ^2 \right]\\& \qquad
         -2 \E\left[ N^{\beta}(t, I,X) \E\left[N^{\beta}(t, I,X) \mid\mathcal{F}_R \right] \right] \\&
         = \E\left[ \E\left[N^{\beta}(t, I,X)^2 \mid \mathcal{F}_R\right] -  \E\left[N^{\beta}(t, I,X) \mid\mathcal{F}_R \right] ^2 \right].
    \end{split}
\end{equation}



Moreover,
\begin{equation}
    \begin{split}\label{JOIR}
         \E&\left[N^{\beta}(t, I,X)^2 \mid \mathcal{F}_R\right] -  \E\left[N^{\beta}(t, I,X) \mid\mathcal{F}_R \right] ^2
         \\& = \sum_{u,u' \in \mathcal{H}^\beta(R,I \cap [0,R])} \E\left[ \sum_{v\in \mathcal{N}^u_{t}}\bbone_{A_{v}}  \sum_{v'\in \mathcal{N}^{u'}_{t}}\bbone_{A_{v'}} \middle| \mathcal{F}_R \right] 
         \\& \qquad \qquad \qquad- \E\left[ \sum_{v\in \mathcal{N}^u_{t}}\bbone_{A_{v}} \middle| \mathcal{F}_R \right]\E\left[ \sum_{v'\in \mathcal{N}^{u'}_{t}}\bbone_{A_{v'}} \middle| \mathcal{F}_R \right].
    \end{split}
\end{equation}

In this latter expression, for $u \neq u'$, since the particles branched before time $R$, their path are conditionally independent on $\mathcal{F}_R$, the terms in the sum cancel each other. Thus \eqref{JOIR} is equal to
\begin{equation}
    \begin{split}
    \E&\left[ \sum_{u \in \mathcal{H}^\beta(R,I \cap [0,R])} \E\left[ \left( \sum_{v\in \mathcal{N}^u_{t}}\bbone_{A_{v}} \right)^2 \middle| \mathcal{F}_R \right] -\E\left[ \sum_{v\in \mathcal{N}^u_{t}}\bbone_{A_{v}} \middle| \mathcal{F}_R \right]^2 \right]
         \\& \leq \E\left[ \sum_{u \in \mathcal{H}^\beta(R,I \cap [0,R])} \sum_{v,v'\in \mathcal{N}^u_{t}}\bbone_{A_{v}}\bbone_{A_{v'}}   \right].
    \end{split}
\end{equation}
\end{proof}



We now proceed to prove Proposition~\ref{L2boundlocancestor}.
\begin{proof}[Proof of Proposition \ref{L2boundlocancestor}]
Let $x=x_t, K, r, R=R_t$ be such that, as $t\to \infty$, \( x_t \to \infty \) with \( x_t = o(t) \), \( R_t \to \infty \), $2r\leq R_t$, $3r \leq t$ and $K\in \mathbb{R}$. Proposition \ref{L2boundlocancestor} states that 
\begin{equation}
    \frac{\E\left[\#\left\{(u,v) \in \mathcal{N}^{\beta_t}(t,[r,t],[x-K,x])^2 : Q_t(u,v) >R_t \right\}\right]}{f(t,x)^2}\xrightarrow[t\rightarrow \infty]{}0,
\end{equation}
where $\beta_t(s)=s\frac{m_t}{t}$.

In the following proof, $C$ denotes a generic constant whose value may change from line to line. By the many-to-two lemma \ref{manyto2}, we have that
\begin{equation}\label{fdfo}
    \begin{aligned}
        \E&\left[\#\left\{(u,v) \in \mathcal{N}^\beta(t,[r,t],[x-K,x])^2, u \neq v : Q_t(u,v) > R \right\}\right] \\
        &= \int_R^t 2e^{2t-\gamma} \Prob\left( 
            \begin{aligned}
                &B^{1,\gamma}_t \in [m_t-x, m_t-x+K], \ \forall s \in [r,t], \, B^{1,\gamma}_s \leq \frac{m_t}{t}s, \\
                &B^{2,\gamma}_t \in [m_t-x, m_t-x+K], \ \forall s \in [r,t], \, B^{2,\gamma}_s \leq \frac{m_t}{t}s
            \end{aligned}
        \right) \mathop{}\!\mathrm{d}\gamma \\
        &= \int_R^t 2e^{2t-\gamma} \int_{-\infty}^{\frac{m_t}{t}\gamma} \phi_\gamma(y) \Prob\left(\forall s \in [r, \gamma], \, B_s \leq \frac{m_t}{t}s \mid B_\gamma = y \right) \\
        &\qquad \times \Prob\left( m_t - B_{t-\gamma} \in [x-K,x], \, \forall s \in [\gamma,t], \, B_{s-\gamma} \leq \frac{m_t}{t}s \mid B_0 = y \right)^2 \mathop{}\!\mathrm{d}y \mathop{}\!\mathrm{d}\gamma \\
        &= \int_R^t 2e^{2t-\gamma} \int_{0}^{+\infty} \phi_\gamma\left(\frac{m_t}{t}\gamma - y\right) \Prob\left(\forall s \in [r, \gamma], \, B_s \geq 0 \mid B_\gamma = y \right) \\
        &\qquad \times \Prob_{-y}\left( \frac{m_t}{t}(t-\gamma) - B_{t-\gamma} \in [x-K,x], \, \forall s \in [0,t-\gamma], \, B_s \leq \frac{m_t}{t}s \right)^2 \mathop{}\!\mathrm{d}y \mathop{}\!\mathrm{d}\gamma,
    \end{aligned}
\end{equation}

where for $\gamma \in \left[0,t\right]$, $(B^{(1,\gamma)}_s)_{s\in \mathbb{R}}$ and $(B^{(2,\gamma)}_s)_{s\in \mathbb{R}}$ are coupled Brownian motions sharing a common path until time $\gamma$, then evolving independently thereafter.

We bound the first two terms in the last integral from equation \eqref{fdfo}. For the first term, we have that,
\begin{equation}
    \phi_\gamma\left(\frac{m_t}{t}\gamma-y\right)= \frac{1}{\sqrt{2 \pi \gamma}}e^{-\frac{m_t^2}{2t^2}\gamma}e^{\frac{m_t}{t}y}e^{- \frac{y^2}{2\gamma}} \leq \frac{1}{\sqrt{2 \pi \gamma}}e^{-\frac{m_t^2}{2t^2}\gamma}e^{\frac{m_t}{t}y}.
\end{equation}

For the second term, applying Lemma \ref{reflectionprinciple0}, we derive
\begin{equation} \label{eq:84}
    \Prob\left(\forall s \in [r, \gamma] \, , \, B_s \geq 0 \mid B_\gamma = y \right)= 1- 2\Prob(B_r \leq 0 \mid B_\gamma = y ).
\end{equation}
Denote $\mu = \frac{y r}{\gamma}$ and $\sigma^2= \frac{r(\gamma - r)}{\gamma}$. Let $G$ be a standard Gaussian random variable, then the the quantity \eqref{eq:84} is equal to
\begin{equation}
    \Prob\left( |G| \leq \frac{\mu}{\sigma}\right) \leq \sqrt{\frac{2}{ \pi}} \frac{\mu}{\sigma}=\sqrt{\frac{2}{ \pi}} \frac{y \sqrt{r}}{\sqrt{\gamma(\gamma - r)}}.
\end{equation}

We partition the domain of integration in \eqref{fdfo} into several subdomains.
We start with several easy cases: First, \eqref{fdfo} on $\gamma \geq t-1$ and $y \leq x+1$ is smaller than
\begin{equation}\label{fdfoy2}
    \begin{split}
        &\int_{t-1}^t 2e^{2t-\gamma}\int_{0}^{x+1}\phi_\gamma\left(\frac{m_t}{t}\gamma-y\right)\Prob\left(\forall s \in [r, \gamma] \, , \, B_s \geq 0 \mid B_\gamma = y \right)\mathrm{d}y\mathrm{d}\gamma
        \\& \leq \int_{t-1}^t 2e^{2t-\gamma}\int_{0}^{x+1}\frac{1}{\pi \gamma}e^{-\frac{m_t^2}{2t^2}\gamma}e^{\frac{m_t}{t}y}\frac{y \sqrt{r}}{\sqrt{\gamma - r}}\mathrm{d}y\mathrm{d}\gamma.
    \end{split}
\end{equation}
Using that $\frac{m_t^2}{2t^2}\gamma\geq  \gamma -\frac{3\gamma \log(t)}{2t}$ and $R\geq 2r$, \eqref{fdfoy2} is bounded by
\begin{equation}\label{fdfoyy}
    \begin{split}
        &C\sqrt{r}\int_{t-1}^t \frac{e^{\frac{3\gamma \log(t)}{2t}}}{\gamma^{\frac{3}{2}}}\int_{0}^{x+1}ye^{\frac{m_t}{t}y}\mathrm{d}y\mathrm{d}\gamma\leq C\sqrt{r}\int_{0}^{x+1}ye^{\frac{m_t}{t}y}\mathrm{d}y
        \\& \leq  C\sqrt{r}xe^{\frac{m_t}{t}x}.
    \end{split}
\end{equation}
Secondly, the RHS of \eqref{fdfo} restricted to $\gamma \geq t-1$ and $y\geq 1+x$ is smaller than
\begin{equation}\label{fdfoy3}
    \begin{split}
        &\int_{t-1}^t 2e^{2t-\gamma}\int_{x+1}^{\infty}\phi_\gamma\left(\frac{m_t}{t}\gamma-y\right)\Prob\left(\forall s \in [r, \gamma] \, , \, B_s \geq 0 \mid B_\gamma = y \right)\\
        & \qquad \qquad \Prob_{-y}\left( B_{t- \gamma} \geq \frac{m_t}{t}(t- \gamma)-x\right) ^2  \mathrm{d}y\mathrm{d}\gamma
        \\& \leq C \sqrt{r} \int_{t-1}^t e^{2t-\gamma}\int_{x+1}^{\infty}\frac{1}{ \gamma^{\frac{3}{2}}}e^{-\frac{m_t^2}{2t^2}\gamma}e^{\frac{m_t}{t}y}y\left(e^{-\frac{m_t^2(t-\gamma)}{2t^2}}e^{\frac{m_t}{t}(x-y)}e^{-\frac{(x-y)^2}{2(t-\gamma)}}\right)^2\mathrm{d}y\mathrm{d}\gamma
        \\& \leq C \sqrt{r}e^{2\frac{m_t}{t}x} \int_{t-1}^t \int_{x+1}^{\infty}e^{-\frac{m_t}{t}y}y\mathrm{d}y\mathrm{d}\gamma\leq C \sqrt{r}xe^{\frac{m_t}{t}x}.
    \end{split}
\end{equation}
The bounds obtained in equations \eqref{fdfoyy} and \eqref{fdfoy3}, when renormalised by $f(t,x)^2$, both tends to $0$ as $x,t\to \infty$ with  $x=o(t)$. Therefore when $\gamma \geq t-1$, \eqref{fdfo} yields a negligible contribution.

We furthermore bound \eqref{fdfo} when $\frac{m_t}{t}(t-\gamma) +y -x \leq 0$ by
\begin{equation}\label{fdfo2}
    \begin{split}
        &\int_R^t 2e^{2t-\gamma}\int_{0}^{+\infty}\phi_\gamma\left(\frac{m_t}{t}\gamma-y\right)\Prob\left(\forall s \in [r, \gamma] \, , \, B_s \geq 0 \mid B_\gamma = y \right)
        \\&\Prob_{-y}\left( B_{t- \gamma} \geq \frac{m_t}{t}(t- \gamma)-x , \, \forall s \in [0,t-\gamma] \, , \, B_{s} \leq \frac{m_t}{t}s \right) ^2  \bbone_{\frac{m_t}{t}(t-\gamma) +y -x \leq 0}\mathrm{d}y\mathrm{d}\gamma
        \\& \leq \int_R^t 2e^{2t-\gamma}\int_{0}^{+\infty}\frac{1}{\pi \gamma}e^{-\frac{m_t^2}{2t^2}\gamma}e^{\frac{m_t}{t}y}\frac{y \sqrt{r}}{\sqrt{\gamma - r}}\bbone_{\frac{m_t}{t}(t-\gamma) +y -x \leq 0}\mathrm{d}y\mathrm{d}\gamma.
    \end{split}
\end{equation}
Using again that $R\geq 2r$, $\frac{m_t^2}{2t^2}\gamma\geq  \gamma -\frac{3\gamma \log(t)}{2t}$ and that on $\{\frac{m_t}{t}(t-\gamma) +y -x \leq 0\}$, $y\leq x$, $\gamma \geq t-x$, \eqref{fdfo2} is bounded by
\begin{equation}\label{fdfo3}
    \begin{split}
    &\leq C \sqrt{r} \int_{t-x}^t e^{2(t-\gamma) +\frac{3\gamma \log(t)}{2t}}\int_{0}^{x}\frac{1}{ \gamma^{\frac{3}{2}}}ye^{\frac{m_t}{t}y}\bbone_{\frac{m_t}{t}(t-\gamma) +y -x \leq 0}\mathrm{d}y\mathrm{d}\gamma\\
    &\leq C \sqrt{r} \int_{t-x}^t e^{\left(2- \frac{m_t^2}{t^2}\right)(t-\gamma)}\int_{0}^{x}ye^{\frac{m_t}{t}\left(y+ \frac{m_t}{t}(t-\gamma)\right)}\bbone_{\frac{m_t}{t}(t-\gamma) +y -x \leq 0}\mathrm{d}y\mathrm{d}\gamma\\
     &\leq C \sqrt{r} e^{\frac{m_t}{t}x}\int_{t-x}^t e^{\frac{3x\log(t)}{t}}\int_{0}^{x}y\mathrm{d}y\mathrm{d}\gamma\leq C \sqrt{r} e^{\frac{m_t}{t}x}x^3e^{\frac{3x\log(t)}{t}}.
    \end{split}
\end{equation}
When renormalized by $f(t,x)^2$, the ratio is bounded by 
\begin{equation*}
    C\sqrt{r} x \exp\left( -\frac{m_t}{t}x + \frac{x^2}{t} + \frac{3x\log(t)}{t} \right) = C\sqrt{r} x \exp\left( x \left( -\frac{m_t}{t} + \frac{x}{t} + \frac{3\log t}{t} \right) \right).
\end{equation*}
Since $x = o(t)$, the terms $\frac{x}{t}$ and $\frac{3\log t}{t}$ both converge to $0$. Thus, the exponent is dominated by the strictly negative linear term $-\frac{m_t}{t}x \sim -\sqrt{2}x$. Therefore, when $\frac{m_t}{t}(t-\gamma) +y -x \leq 0$, \eqref{fdfo} yields a negligible contribution.

Thus, it remains to bound \eqref{fdfo} when $\gamma \leq t-1$ and $\frac{m_t}{t}(t-\gamma) +y -x > 0$.
We first bound the third term appearing in the last integral in equation \eqref{fdfo}. By monotonicity and Lemma \ref{bmbar}, we have that
\begin{equation}
    \begin{split}
         &\Prob_{-y}\left(  \frac{m_t}{t}(t- \gamma) - B_{t- \gamma} \in [x-K,x] , \, \forall s \in [0,t-\gamma] \, , \, B_{s} \leq \frac{m_t}{t}s \right)
         \\& \qquad \leq \Prob_{-y}\left(  \frac{m_t}{t}(t- \gamma) - B_{t- \gamma} \in [x-K,x] \right) 
         \\&  \qquad \qquad \Prob_{0,-y}^{t-\gamma,\frac{m_t}{t}(t- \gamma)-x }\left(\forall s \in [0,t-\gamma] \, : B_{s} \leq \frac{m_t}{t}s \right)
         \\&  \qquad =\Prob_{-y}\left(  \frac{m_t}{t}(t- \gamma) - B_{t- \gamma} \in [x-K,x] \right) \Prob_{0,-y}^{t-\gamma,-x }\left(\forall s \in [0,t-\gamma] \, : B_{s} \leq 0\right)
        \\&  \qquad = \Prob_{-y}\left(  \frac{m_t}{t}(t- \gamma) - B_{t- \gamma} \in [x-K,x]\right) \left(1- e^{-2\frac{xy}{t- \gamma}}\right)
        \\& \qquad \leq \Prob_{-y}\left( \frac{m_t}{t}(t- \gamma) - B_{t- \gamma} \in [x-K,x] \right)2\frac{xy}{t- \gamma}.
    \end{split}
\end{equation}

Moreover, on the domain $\mathcal{A} \coloneqq \{(y,\gamma) : \frac{m_t}{t}(t-\gamma) + y > x\}$, we have:
\begin{equation}
    \begin{aligned}
        \Prob_{-y}&\left(\tfrac{m_t}{t}(t-\gamma) - B_{t-\gamma} \in [x-K,x]\right) \\
        &= \frac{1}{\sqrt{2\pi(t-\gamma)}} \int_{\frac{m_t}{t}(t-\gamma) + y - x}^{\frac{m_t}{t}(t-\gamma) + y - x + K} 
            e^{-\frac{z^2}{2(t-\gamma)}} \mathrm{d}z \\
        &\leq \frac{K}{\sqrt{2\pi(t-\gamma)}} \exp\left(-\frac{\left(\tfrac{m_t}{t}(t-\gamma) + y - x\right)^2}{2(t-\gamma)}\right) \\
        &= \frac{K}{\sqrt{2\pi(t-\gamma)}} \exp\left(
            -\tfrac{m_t^2(t-\gamma)}{2t^2} + \tfrac{m_t}{t}(x-y) - \tfrac{(x-y)^2}{2(t-\gamma)}
        \right).
    \end{aligned}
\end{equation}

We now insert all these estimates back in \eqref{fdfo} and bound \eqref{fdfo} on the domain $\mathcal{A} $ and when $\gamma \leq t-1$ by

\begin{equation}\label{fge3}
\begin{split}
    CK^2\int_R^{t-1} &e^{2t-\gamma}\int_{0}^{+\infty}\bbone_{\mathcal{A}}\frac{1}{\sqrt{\gamma}}e^{-\frac{m_t^2}{2t^2}\gamma}e^{\frac{m_t}{t}y}\frac{y \sqrt{r}}{\sqrt{\gamma(\gamma - r)}}
        \\& \qquad \qquad \left( \frac{xy}{t- \gamma}\right)^2 \left(\sqrt{\frac{1}{t- \gamma}}e^{-\frac{m_t^2(t-\gamma)}{2t^2}}e^{\frac{m_t}{t}(x-y)}e^{-\frac{(x-y)^2}{2(t-\gamma)}}\right)^2  \mathrm{d}y\mathrm{d}\gamma
         \\& = CK^2 x^2e^{2\frac{m_t}{t}x}\sqrt{r}\int_R^{t-1}\frac{e^{2t-\gamma -\frac{m_t^2}{t}+\gamma\frac{m_t^2}{2t^2}}}{\gamma\sqrt{\gamma-r}(t-\gamma)^{3}}
         \\& \qquad \qquad\int_{0}^{+\infty}\bbone_{\mathcal{A}}y^3e^{-\frac{m_t}{t}y}e^{-\frac{(x-y)^2}{t-\gamma}}\mathrm{d}y\mathrm{d}\gamma.
\end{split}
\end{equation}

We further partition the analysis into two subcases. First, if $\gamma \leq t-x/10$, the RHS of \eqref{fge3} is bounded by: 
\begin{equation}\label{fge4}
\begin{split}
        & CK^2 x^2e^{2\frac{m_t}{t}x}e^{-\frac{x^2}{t}}\sqrt{r}\int_R^{t-\frac{x}{10}} \frac{e^{3\log(t)-\frac{3}{2}\frac{\gamma}{t}\log(t)}}{\gamma\sqrt{\gamma-r}(t-\gamma)^{3}}\int_{0}^{+\infty}y^3e^{-\left(\frac{m_t}{t}-2\frac{x}{t}\right)y}\mathrm{d}y\mathrm{d}\gamma.
\end{split}
\end{equation}

Since \(\left(\tfrac{m_t}{t} - 2\tfrac{x}{t}\right) \geq 1\) for \(x \leq t/10\) and \(t \geq 1\), there exists \(B > 0\) such that:
\[
\int_0^{+\infty} y^3 e^{-\left(\tfrac{m_t}{t} - 2\tfrac{x}{t}\right)y} \mathrm{d}y \leq B.
\]

Moreover, under the condition \(R \geq 2r\) (implying \(\gamma \geq 2r\)), we have \(\gamma \leq 2(\gamma - r)\). This yields:
\[
\int\limits_R^{t-\frac{x}{10}} \frac{e^{3\log t - \tfrac{3\gamma}{2t}\log t}}{\gamma\sqrt{\gamma-r}(t-\gamma)^3} \mathrm{d}\gamma 
\leq C \int\limits_R^{t-\frac{x}{10}} \frac{e^{3\log t - \tfrac{3\gamma}{2t}\log t}}{\gamma^{3/2}(t-\gamma)^3} \mathrm{d}\gamma.
\]

There exists \( x_0 > 0 \) such that for all \( t \geq x_0 \) and \( \gamma \in [R, t - x_0/10] \),
\begin{equation*}
    \log t \leq \tfrac{\gamma}{t}\log t + \log(t - \gamma).
\end{equation*}
Therefore, for $x$ large enough,
\begin{equation}
    \begin{split}
         \int_R^{t-\frac{x}{10}} \frac{e^{3\log(t)-\frac{3}{2}\frac{\gamma}{t}\log(t)}}{\gamma^{\frac{3}{2}}(t-\gamma)^{3}}\mathrm{d}\gamma\leq \int_R^{t-\frac{x}{10}} \frac{e^{\frac{3}{2}\log(t)}}{\gamma^{\frac{3}{2}}(t-\gamma)^{\frac{3}{2}}}\mathrm{d}\gamma,
    \end{split}
\end{equation}
which tends to 0 with $R\to \infty$ and $x\to \infty$.
Thus \eqref{fge4}, renormalised by $f(t,x)^2$, tends to $0$.

 Secondly, for $(y, \gamma) \in \mathcal{A}$, if $\gamma > t - x/10$ then $y \geq 4x/5$. Thus, for $x$ large enough, the RHS of \eqref{fge3} is smaller than
\begin{equation}
\begin{split}
        & CK^2 x^2e^{2\frac{m_t}{t}x}\sqrt{r}\int_{t-\frac{x}{10}}^{t-1} e^{\frac{3}{2}(1-\frac{\gamma}{t})\log(t)}\int_{\frac{4x}{5}}^{+\infty}y^3e^{-\frac{m_t}{t}y}\mathrm{d}y\mathrm{d}\gamma\\
        & \leq C K^2 x^2e^{\frac{7}{5}\frac{m_t}{t}x}\sqrt{r}\int_{t-\frac{x}{10}}^{t-1} e^{\frac{3}{2}(1-\frac{\gamma}{t})\log(t)}\mathrm{d}\gamma\\
        &\leq Cx^3e^{\frac{7}{5}\frac{m_t}{t}x}\sqrt{r}e^{\frac{3}{20}\frac{x\log(t)}{t}}.
\end{split}
\end{equation}

When this bound is renormalised by $f(t,x)^2$, the ratio is bounded by
\begin{equation*}
    C \sqrt{r} x \exp\left( -\frac{3}{5}\frac{m_t}{t}x + \frac{x^2}{t} + \frac{3x\log(t)}{20 t} \right) = C \sqrt{r} x \exp\left( x\left( -\frac{3}{5}\frac{m_t}{t} + \frac{x}{t} + \frac{3\log t}{20 t} \right) \right).
\end{equation*}
Since $x = o(t)$, the terms $\frac{x}{t}$ and $\frac{3\log t}{20 t}$ both converge to $0$. Thus, the exponent is dominated by the strictly negative linear term $-\frac{3}{5}\frac{m_t}{t}x \sim -\frac{3\sqrt{2}}{5}x$. Therefore, the renormalised bound tends to $0$, and when $\frac{m_t}{t}(t-\gamma) +y -x > 0$, \eqref{fdfo} yields a negligible contribution.
\end{proof}

\bibliographystyle{plain}
\bibliography{biblio}
\end{document}